\documentclass[reqno,11pt]{amsart}
\usepackage{amsmath, amssymb,  psfrag, multicol, float, graphicx, amsthm, mathrsfs, mathtools}
\usepackage[labelfont=bf]{caption}
\usepackage[all,cmtip]{xy}
\usepackage[dvips]{color}
\usepackage{pdfpages, enumerate}
\usepackage[numbers]{natbib}
\usepackage[mathscr]{euscript}
\let\euscr\mathscr
\newtheorem{thm}{Theorem}[section]
\newtheorem{lemma}[thm]{Lemma}

\newtheorem{corollary}[thm]{Corollary}

\newtheorem{rmk}[thm]{Remark}

\newtheorem{exmpl}[thm]{Example}
\newtheorem{defn}[thm]{Definition}

\begin{document}

\title[Quantum Steenrod squares and the equivariant pair-of-pants]{Quantum Steenrod squares and the equivariant pair-of-pants in symplectic cohomology}

\author{Nicholas Wilkins}
\address{Address on date of upload: School of Mathematics, University of Bristol, Bristol BS8 1UG, UK, and Heilbronn Institute for Mathematical Research, Bristol, UK.}
\email{Email on date of upload: nicholas.wilkins@bristol.ac.uk}
\date{version: June 4, 2019. Upload date: \today}

\begin{abstract}
		We relate the quantum Steenrod square to Seidel's equivariant pair-of-pants product for open convex symplectic manifolds that are either monotone or exact, using an equivariant version of the PSS isomorphism. We proceed similarly for $\mathbb{Z}/2$-equivariant symplectic cohomology, using an equivariant version of the continuation and $c^*$-maps. We prove a symplectic Cartan relation, pointing out the difficulties in stating it. We give a nonvanishing result for the equivariant pair-of-pants product for some elements of $SH^*(T^* S^n)$. We finish by calculating the symplectic square for the negative line bundles $M = \text{Tot}(\mathcal{O}(-1) \rightarrow \mathbb{CP}^m)$, proving an equivariant version of a result due to Ritter.
\end{abstract}

\maketitle

\section{Introduction}
Throughout, the term ``equivariant" will refer to $\mathbb{Z}/2$-equivariance. All cohomology will have $\mathbb{Z}/2$-coefficients. We will also interchangeably use Morse cohomology and singular cohomology in all of our constructions.

In this paper we discuss the relationship between the construction of the quantum Steenrod square in \cite{me} and the equivariant pair-of-pants product due to Seidel. They are both generalisations of the Steenrod square on a topological space $M$, which is an additive homomorphism $$Sq: H^n(M) \rightarrow H^{2n}_{eq}(M) := (H^{\bullet}(M)[[h]])^{2n},$$ where $h$ is a variable of degree $1$. The quantum Steenrod square was defined for closed monotone symplectic manifolds in \cite{me}, but a modified definition may be viable more generally any time quantum cohomology is well defined, i.e. weakly monotone symplectic manifolds. That being said, in this paper we restrict to the case when $M$ is either monotone or exact, and convex at infinity. The equivariant pair-of-pants product was defined for exact symplectomorphisms of Liouville domains by Seidel in \cite{seidel}, but here we use the identity symplectomorphism and a more general symplectic manifold, i.e. open convex symplectic manifolds as in Definition \ref{defn:convexity}. 

We begin with a fairly detailed preliminary section, which brings together the important background material that will be used. We will define the quantum cohomology in Section \ref{subsec:quantcupprod2}, recalling that for a general weakly monotone symplectic manifold $M$ the quantum cohomology $QH^*(M,\omega)= H^*(M, \Lambda)$ as a vector space, using a Novikov ring $\Lambda$. The cup product is deformed by quantum contributions from counting 3-pointed genus zero Gromov-Witten invariants, i.e. counting certain $J$-holomorphic spheres in $M$ where $J$ is an almost complex structure on $M$ compatible with $\omega$. We will denote this product $*$. We introduce in Section \ref{subsec:SqQviaMorse2} the quantum Steenrod square, which is an operation $$Q\mathcal{S} : QH^{*}(M) = H^{*}(M)[t] \rightarrow H^{*}(M)[t][[h]] = QH^{*}(M)[[h]].$$ Throughout we will denote $QH^*_{eq}(M) := (QH^{\bullet}(M)[[h]])^*$, to adhere to the notation of \cite{seidel}. The quantum Steenrod square differs from the classical Steenrod square because it is not a ring morphism: specifically, the obvious analogue of the Cartan relation \begin{equation} \label{equation:classcartan} Sq(x \cup y) = Sq(x) \cup Sq(y) \end{equation} does not hold. However, there is a correction term that can be computed, yielding a quantum Cartan relation $$Q \mathcal{S}(x*y)=Q \mathcal{S}(x) * Q \mathcal{S}(y) + \sum_{i,j} q_{i,j}(W_0 \times D^{i,+})(x,y)h^i,$$ for a correction term $q_{i,j}$ as in \cite[Theorem 1.2]{me}. We will also recall in Section \ref{subsec:equivHF} the notion of equivariant Floer cohomology $HF^*_{eq}(H)$ for a Hamiltonian $H: M \rightarrow \mathbb{R}$, from \cite{seidel}, which is a way of generalising the $\mathbb{Z}/2$-equivariant cohomology to Floer theory. In the same paper Seidel defined a generalisation of the Steenrod square, which throughout we call the symplectic square: $$P \mathcal{S}: HF^*(H) \rightarrow HF^*_{eq}(2 \cdot H).$$ We describe this in Section \ref{subsec:eqpop}. In Section \ref{subsec:equivariantgluing} we give a sketch of how the standard gluing and compactness arguments generalise to the equivariant case.

In Section \ref{subsec:equivariantcontinuationmaps}, we will define the equivariant continuation maps $\phi_{eq,H,H'}$, which are homomorphisms $HF^*_{eq}(2 \cdot H) \rightarrow HF^*_{eq}(2 \cdot H')$. Our convex symplectic manifolds $M$ are split into a compact symplectic manifold $C$, with $\dim C = \dim M$ such that $\partial C$ is contact, and a collar neighbourhood symplectomorphic to $\partial C \times [1,\infty)$. To be radial at infinity is to say that there is some $R \ge 1$ such that $H(z,r) = \lambda_H \cdot r$ for $(z,r) \in \partial C \times [1,\infty)$ and $r \ge R$. The equivariant continuation maps are defined when $H, H'$ are radial at infinity with $\lambda_H \le \lambda_{H'}$.

We recall in Section \ref{subsec:pssisomprelim} the ring isomorphism due to Piunikhin-Salamon-Schwarz \cite{PSS}, denoted $\Psi: QH^{*}(M) \xrightarrow{\cong} HF^{*}(H)$ is a homomorphism from quantum cohomology to Floer cohomology. The PSS isomorphism $\Psi$ is defined for a small Hamiltonian $H$, where ``small" will be made precise in the referenced section. In Section \ref{sec:eqpssisom} we will construct an equivariant version of the PSS-isomorphism, $$\Psi_{eq}: QH^*_{eq}(M) \rightarrow HF^*_{eq}(2 \cdot H),$$ where $2 \cdot H$ is ``small". In Section \ref{sec:qsandeqpopclosed} we will prove the main result of this paper:
\begin{thm}
\label{thm:qsseqpopintertwine} Let $M$ be a convex symplectic manifold, and $H: M \rightarrow \mathbb{R}$ a $C^2$-small Hamiltonian. Then the following diagram commutes:

\begin{equation}\label{equivpssisom}
\xymatrix{
QH^*(M)
\ar@{->}^-{Q\mathcal{S}}[r]
\ar@{->}^-{\cong}_-{\Psi}[d]
&
QH^*_{eq}(M)
\ar@{->}^-{\cong}_-{\Psi_{eq}}[d]
\\ 
HF^*(H)
\ar@{->}^-{P \mathcal{S}}[r]
&
HF^*_{eq}(2 \cdot H)
}
\end{equation}

\end{thm}

In Section \ref{sec:qsandeqpopopen}, we will use the continuation maps from Section \ref{subsec:equivariantcontinuationmaps} to define the equivariant symplectic cohomology $SH_{eq}^*(M)$ as the direct limit of $HF_{eq}^*(2 \cdot H)$ over all $H$ that are radial at infinity, using the maps $\phi_{eq,H,H'}$. Composing the equivariant PSS isomorphism $QH_{eq}^*(M) \xrightarrow{\cong} HF_{eq}^*(2 \cdot H)$ for a $C^2$-small Morse function $H$ with the map $HF_{eq}^*(2 \cdot H) \rightarrow SH_{eq}^*(M)$ to the direct limit, there is a map  $c_{eq}^* : QH_{eq}^*(M) \rightarrow SH_{eq}^*(M)$. We observe that the symplectic square commutes with the equivariant continuation maps:
\begin{lemma}
\label{lemma:pairofpantsandcontin}
$$\phi_{eq,H,H'} \circ P \mathcal{S} = P \mathcal{S} \circ \phi_{H,H'}.$$
\end{lemma}

This lemma allows us to define a symplectic square $P \mathcal{S}:SH^*(M) \rightarrow SH^*_{eq}(M)$, and from the previous results obtain the following corollary:

\begin{corollary}
\label{corollary:qsseqpopintertwine2}
The following diagram commutes:

\begin{equation}\label{equivcstar}
\xymatrix{
QH^*(M)
\ar@{->}^-{Q\mathcal{S}}[r]
\ar@{->}_-{c^*}[d]
&
QH^*_{eq}(M)
\ar@{->}_-{c^*_{eq}}[d]
\\ 
SH^*(M)
\ar@{->}^-{P \mathcal{S}}[r]
&
SH^*_{eq}(M)
}
\end{equation}

\end{corollary} 

In Section \ref{sec:symplcartan} we discuss an attempt to construct a symplectic version of the Cartan relation. At the beginning of the referenced section we see that an immediate generalisation of the Cartan relation does not make sense because there is no obvious pair-of-pants type product on $SH^*_{eq}(M)$. This is because the pair-of-pants is too rigid to have a holomorphic involution that rotates each of the three cylindrical ends halfway. This is disappointing because the classical version of the Cartan relation is a very useful computational tool, as is the quantum Cartan relation. We are partially saved because while $SH^*_{eq}(M)$ is not a ring, it is a module over $SH^*(M)$, which is encoded in an operation $$P \mathcal{S}': SH^*(M) \otimes SH^*_{eq}(M) \rightarrow SH^*_{eq}(M),$$ and we can prove:

\begin{thm}[Symplectic Cartan relation]
\label{thm:symplecticcartan}

	$$P \mathcal{S}(x * y) = P \mathcal{S}'(x;P \mathcal{S}(y)).$$
\end{thm}

Here $*$ denotes the pair-of-pants product on $SH^*(M)$. We can see that this is of a similar form to the classical Cartan relation (in fact, one could rephrase the classical Cartan relation in terms of Theorem \ref{thm:symplecticcartan}). Indeed, this shows that in fact $P \mathcal{S}$ is completely determined by $P\mathcal{S}(1)$ and the operation $P\mathcal{S}'$, using that $P \mathcal{S}(x*y) = P \mathcal{S}'(x*y, P \mathcal{S}(1))$. In practise computing $P \mathcal{S}'$ is as difficult as computing $P \mathcal{S}$.

In Section \ref{subsec:calcs} we will demonstrate the nonvanishing of $P \mathcal{S}(a)$ for half of the additive generators $a \in SH^*(T^* S^n)$.

We finish this paper by considering $M$ being the total space of a negative line bundle over a closed symplectic manifold. The specific case will be $M = \text{Tot}(\mathcal{O}(-1) \rightarrow \mathbb{CP}^m)$. We begin with the work of Ritter in \cite{ritterFTNLB} and modify to the case of equivariant cohomology. In the given case, Ritter proved that the $c^*$-map induces an isomorphism $c^*: QH^*(M)/\ker r^k \xrightarrow{\cong} SH^*(M)$ for a particular linear homomorphism $r: QH^*(M) \rightarrow QH^*(M)$. We will use Diagram \eqref{equivcstar} to show that $c^*_{eq}$ descends to a map $$c^*_{eq}: QH_{eq}^*(M)/(Q \mathcal{S}(\ker r^k)) \rightarrow SH^*_{eq}(M)$$ in this case.  It is immediate from Corollary \ref{corollary:qsseqpopintertwine2} that \begin{equation} \label{equation:finaleq} P \mathcal{S} = c^*_{eq} \circ Q \mathcal{S} \circ (c^*)^{-1}, \end{equation} and we will prove (where $\Lambda$ is an appropriate Novikov field) that
\begin{thm}
\label{thm:finalthm}
Let $M= \text{Tot}(\mathcal{O}(-1) \rightarrow \mathbb{CP}^m)$, and $x$ generates $QH^*(M)$, where $T$ is the quantum variable with $|T|= 2m$. Then:
$$c^*_{eq}: QH_{eq}^*(M)/Q \mathcal{S}(x^m + T) \cdot \Lambda[[h]] \rightarrow SH^*_{eq}(M)$$ is an isomorphism.
\end{thm}
Together with Equation \ref{equivcstar}, this yields the following commutative diagram:

\begin{equation}\label{equivFTNLB1}
\xymatrix{
QH^*(M)/ (x^m + T) \cdot \Lambda
\ar@{->}^-{Q\mathcal{S}}[r]
\ar@{->}^-{\cong}_{c^*}[d]
&
QH^*_{eq}(M)/ Q \mathcal{S}(x^m + T) \cdot \Lambda[[h]]
\ar@{->}_-{c^*_{eq}}^{\cong}[d]
\\ 
SH^*(M)
\ar@{->}^-{P \mathcal{S}}[r]
&
SH^*_{eq}(M)
}
\end{equation}

This will allow us to calculate the symplectic square in terms of the quantum Steenrod square. For $M = \text{Tot}(\mathcal{O}(-n) \rightarrow \mathbb{CP}^m)$ we can use the quantum Cartan relation, \cite[Theorem 1.2]{me}, to calculate the quantum Steenrod squares for $M$ similarly to in \cite[Section 6.1]{me}. Using Equations \eqref{equation:finaleq} and the calculation of the quantum Steenrod square, we can then calculate the symplectic square.

\subsection*{Acknowledgements}
I thank my supervisor Alexander Ritter for his guidance and support and I thank Paul Seidel for suggesting this project and for helpful conversations. I thank Dominic Joyce and Ivan Smith for their helpful comments on my thesis (and therefore on this paper). This work was supported by an EPSRC grant, reference: EP/M508111/1, and partially supported by the Simons Foundation, through a Simons Investigator grant (PI: Paul Seidel).

This work forms part of my Ph.D. thesis.

\section{Preliminaries}
		\subsection{Equivariant Floer cohomology}
		\label{subsec:equivHF}
		A full treatment of this due to Seidel can be found in \cite[Section (4b)]{seidel}. The cited paper focuses on the case where $M$ is a Liouville domain, with \cite[Section 7]{seidel} discussing the nonexact case. We will be considering the specific case where our symplectic manifold is either monotone or exact (we use this to deal with bubbling). We are only interested in the cases where the symplectomorphism is Hamiltonian. For a Hamiltonian symplectomorphism $\phi^H$, we rephrase the definition in terms of Hamiltonian rather than fixed point Floer cohomology. 

Recall that the fixed-point Floer cohomology $HF^*(\phi_1^H)$ for a Hamiltonian symplectomorphism, $\phi_1^H$, was defined in \cite{floerfixedpoints}. Using an unwrapping isomorphism one can show that the fixed-point and Hamiltonian Floer cohomologies are the same, i.e. $HF^*(\phi_1^H) \cong HF^*(H)$. Specifically, given a $1$-periodic Hamiltonian $H = H_t$, let $\phi_t = \phi^{H}_t$ be the time $t$ flow of $H_t$. By this we mean the unique symplectomorphism $(\phi_t)_{t \in [0,1]}$ such that $\phi_0 = \text{id}$ and $\partial \phi_t / \partial t = - \nabla H_t$. Suppose that $a$ is a fixed point of $\phi_1$. Then $b(t) = \phi_t(a)$ is a Hamiltonian loop with respect to $H_t$. Further, if $v: \mathbb{R} \times \mathbb{R} \rightarrow M$ is a $\phi_1$-twisted $J$-holomorphic strip, i.e. $v$ is $J$-holomorphic and $v(s,t+1) = \phi_1 v(s,t)$, then let $u(s,t) = \phi_t^{-1} v(s,t)$ for $(s,t) \in \mathbb{R} \times \mathbb{R}$. Then $u(s,t)$ descends to a map $\overline{u}: \mathbb{R} \times \mathbb{R} / \mathbb{Z} \rightarrow M$ that satisfies \begin{equation} \label{equation:floertraj} \dfrac{\partial u}{\partial s} + J \left( \dfrac{\partial u}{\partial t} - X_H \right) = 0. \end{equation} Hence the generators and differentials of the two Floer cohomologies correspond bijectively, hence they are isomorphic.

Let $H$ be a Hamiltonian. Let $$\mathcal{L} = \{ x: \mathbb{R}/2 \mathbb{Z} \rightarrow M | \dot{x} = X_H \}$$ be the space of $2$-periodic Hamiltonian loops with respect to the Hamiltonian $H$. Recall that $x(t)$ is a $2-$periodic Hamiltonian loop for $H$ if and only if $x(2t)$ is a $1$-periodic Hamiltonian loop for $2 \cdot H$. We think of the generators of $CF^{*}(2 \cdot H)$ (working over the Novikov ring $\Lambda$) as being $2$-periodic Hamiltonian loops for $H$, as this will be more useful for the definition of the symplectic square in Section \ref{subsec:eqpop}. Define  $$\mathcal{J}_{\epsilon} = \{ (J_t) : (J_t)  \text{ is an almost compex structure on } M \text{ with } J_{t+\epsilon} = J_{t} \}.$$ There is a $\mathbb{Z}/2$ action $\rho_{*}$ on $\mathcal{J}_2$, with $(\rho_* J)_t = J_{t+ 1}$. The fixed set of $\rho_*$ is $\mathcal{J}_1 \subset \mathcal{J}_2$. The map $\rho_*$ is induced by a map $\rho$ on $\mathcal{L}$, where $(\rho x)(t) = x(t+ 1)$. This induces an involution on the abelian group $CF^{*}(2 \cdot H)$ by action on the basis elements, i.e. $2$-periodic Hamiltonian loops with respect to $H$. However, this involution is not compatible with the Floer differential.

		More concretely, transversality may fail if we chose $J \in \mathcal{J}_1$, and hence we cannot ensure that $\rho_*$ fixes $J$. In the Floer differential we count Floer trajectories, which are smooth maps satisfying Equation \eqref{equation:floertraj}; observe this equation depends upon a choice of $J$. So our differential $d = d_{J}$ depends upon our almost complex structure, and hence is not fixed by $\rho_*$. When $\rho_*$ acts on the abelian groups $CF^* (2 \cdot H)$, the almost complex structure $J \in \mathcal{J}_2$ changes to $\rho_* J$. So $\rho$ induces an isomorphism of chain complexes $$\rho: (CF^* (2 \cdot H), d_{\rho_{*}J}) \cong (CF^* (2 \cdot H), d_{J})$$ but $\rho$ is not a chain involution itself. Precomposing this with a continuation map  $$\Phi: (CF^* (2 \cdot H), d_{J}) \rightarrow (CF^* (2 \cdot H), d_{\rho_{*}J})$$ yields a chain map $$ \rho \circ \Phi: (CF^* (2 \cdot H), d_{J}) \rightarrow (CF^* (2 \cdot H), d_{J}).$$ We notice that $\rho \circ \Phi$ is a chain map, but not an involution. However, it induces an involution on homology. This gives a partial reason behind the difficulties in Section \ref{sec:symplcartan}. We contrast with the quantum equivariant cohomology case, where the $\mathbb{Z}/2$-action on chains is trivial, hence $QH^*_{eq}(M) = QH^*(M)[[h]]$. 

We now construct $HF_{eq}^* (2 \cdot H)$. Given coordinates $x_{i}$ for $i=0,1,2,...$ on $S^{\infty}$, so $$S^{\infty} = \left\{ (x_0,x_1,...) \biggr\rvert \sum x_i^2 = 1 \text{ and only finitely many } x_i \text{ are nonzero} \right\},$$ define a Morse function \begin{equation} \label{equation:sinftymorse} g: S^{\infty} \rightarrow \mathbb{R}, \ g(x) = \sum_{k} k x_{k}^{2} \end{equation} on $S^{\infty}$ using the round metric, with critical points $v^{i,\pm}$ of index $i$. The $\{ v^{i,+} \}$ are the standard Euclidean basis in $\mathbb{R}^{\infty}$ and $v^{i,-} = - v^{i,+}$. Fix an almost complex structure $J \in \mathcal{J}_2$ on $M$. For each $v \in S^\infty$ pick $J_{eq,v} \in \mathcal{J}_2 $ such that:
		\begin{enumerate}
			\item $J_{eq,-v} = \rho_* J_{eq,v}$
			\item In a neighbourhood of $v^{i,+}$ for any $i$, the almost complex structure $J_{eq,v}$ is independent of $v$ and equal to $J$.
			\item $J_{eq,\tau (v)} = J_{eq,v}$ for $\tau : S^\infty \rightarrow S^\infty$ the shift map, i.e. $\tau(x_0, x_1, ...) = (0, x_0, x_1,...)$.
		\end{enumerate}

		Given a negative gradient flowline $w: \mathbb{R} \rightarrow S^\infty$ of $g$ above, with $w(-\infty) = v^{i, \pm}$ and $w(\infty) = v^{0, +}$, we define a domain dependent almost complex structure on $M$, parametrised by $(s,t) \in \mathbb{R} \times \mathbb{R}/2 \mathbb{Z}$, \begin{equation} \label{equation:jweq} J^{w}_{s,t} = J_{eq,w(s),t}. \end{equation} Using $J^w_{s,t}$ we define a Cauchy Riemann Equation, for $x,y$ being $2$-periodic Hamiltonian loops with respect to $H$:

\begin{equation}
\label{equation:eqHF}
              \begin{cases} \begin{array}{ll}
			u : \mathbb{R} \times (\mathbb{R}/ 2\mathbb{Z}) \rightarrow M \\
			\partial_s u + J^w_{s,t} \partial_t u = - \nabla H \\
			{\displaystyle \lim_{s \rightarrow -\infty} u(s,t) = y(t) \text{ if } w(-\infty) = v^{i,+}} \\
			{\displaystyle \lim_{s \rightarrow -\infty} u(s,t) = y(t+ 1) \text{ if } w(-\infty) = v^{i,-}} \\
			{\displaystyle \lim_{s \rightarrow +\infty} u(s,t) = x(t)} \\
                \end{array} \end{cases}
\end{equation}

There is an $\mathbb{R}$ action simultanously translating the $s$ for both $w$ and $u$. Quotienting by this gives a moduli space of ($\mathbb{R}$-equivalence classes of) pairs, $$\mathcal{M}^{i,\sigma}_{eq}(y,x) = \{ [w,u] : \text{ the conditions in } (\ref{equation:eqHF}) \text{ hold, and } w(-\infty) = v^{i,\sigma} \}$$ 

		\begin{defn}
			Define the equivariant differential $$d_{eq} = d_{J} + \sum_{i \ge 1} h^{i} (d^{i,+}_{eq} + d^{i,-}_{eq})$$ where $d_{J}$ is the differential on Hamiltonian Floer Cohomology (with almost complex structure $J$), and $$d^{i,\sigma}_{eq}(x) = \sum_{y} \# \mathcal{M}^{i,\sigma}_{eq}(y,x) y$$ 

			Here $\#$ is the count of isolated solutions in the parametrised moduli space. 

		\end{defn}

		It is shown in \cite[Definition 4.4]{seidel} that $d_{eq}$ is a differential on $$CF^*_{eq}(2 \cdot H) := CF^*(2 \cdot H)[[h]].$$ Let $HF^*_{eq}(2 \cdot H)$ be the cohomology of $(CF^*_{eq}(2 \cdot H), d_{eq}).$ Consider $HF_{eq}^* (2 \cdot H)$ for $2 \cdot H$ a $C^2$-small, time independent Morse function. The elements of $\mathcal{M}^{i,\sigma}_{eq}(y,x)$ become negative gradient flowlines of $H$, and $HF_{eq}^{*} (2 \cdot H)$ becomes  $HM_{\mathbb{Z}/2}^* (M,H)$, the equivariant Morse cohomology of $M$ (with the trivial $\mathbb{Z}/2$-action).

		\begin{rmk}

\begin{enumerate}
\item In practise we need a time dependent perturbation $H = H_t$, although we choose this such that $H_t = H_{t+1}$ (even though the Hamiltonian loops we consider are $2$-periodic). In this case $(2 \cdot H)_t := 2H_{2t}$.
\item We cannot choose $J \in \mathcal{J}_1$ because our moduli spaces are not guaranteed to be regular.
\item We will discuss compactification and gluing in Section \ref{subsec:equivariantgluing}.
\end{enumerate}
		\end{rmk}

		\subsection{The symplectic square and equivariant pair-of-pants product}
		\label{subsec:eqpop}
		The construction of the symplectic square (which is a specific case of the equivariant pair-of-pants product due to Seidel) follows that given in \cite[Section (4c)]{seidel}, although we work specifically in the case of a Hamiltonian symplectomorphism. The symplectic square will be a homomorphism: $$P \mathcal{S} : HF^{*}(H) \rightarrow HF^{2*}_{eq}(2 \cdot H).$$  

		Let $S$ be a 2-to-1 branched cover of the cylinder $\mathbb{R} \times \mathbb{R}/ \mathbb{Z}$, with covering map $\pi$. Specifically, there is a left-hand region parametrised by $$\epsilon^+ : (-\infty, -1] \times \mathbb{R}/2 \mathbb{Z} \rightarrow S$$ and two right hand regions parametrised by $$\delta^\pm : [1, \infty) \times \mathbb{R}/ \mathbb{Z} \rightarrow S.$$ Let $\epsilon^-(s,t) = \epsilon^+(s,t+1)$. The covering involution $\gamma$ on $S$ swaps $\epsilon^\pm$, and swaps $\delta^\pm$. 

				Choose $J_{\text{left},v,s,t} \in \mathcal{J}_2$ for $s < 1$, $v \in S^{\infty}$ such that:
		\begin{enumerate}
			\item $J_{\text{left},-v,s,t} = J_{\text{left},v,s,t+1}$
			\item $J_{\text{left},v,s,t} \in \mathcal{J}_1$ for $s \ge -1$
			\item $J_{\text{left},v,s,t} = J_{\text{eq},v, t}$ for $s \le -2$
			\item $J_{\text{left},\tau v, s, t} = J_{\text{left},v,s,t}$ 
		\end{enumerate}

		Further, choose $J^{\pm}_{\text{right},v,s,t} \in \mathcal{J}_1$ for $s > -1$ such that:
		\begin{enumerate}
			\item $J^{\pm}_{\text{right},-v,s,t} = J^{\mp}_{\text{right},v,s,t}$
			\item $J^{\pm}_{\text{right},v,s,t}$ is independent of $v,s$ for $s \ge 2$
			\item $J^{\pm}_{\text{right},v,s,t} = J_{\text{left},v,s,t}$ for $s \le 1$
			\item $J^{\pm}_{\text{right},\tau v, s, t} = J^{\pm}_{\text{right},v,s,t}$ 
		\end{enumerate}

		Let $w : \mathbb{R} \rightarrow S^{\infty}$ be a negative gradient flowline of $f$. Define $J^{w}_{z}$ for $z \in S$, with $J^w_z$ illustrated in Figure \ref{fig:jpoppic}.
		\begin{enumerate}
			\item $J^w_{z} = J_{\text{left},w(s), s,t} = J^{\pm}_{\text{right},w(s),s,t}$ if $\pi(z) = (s,t) \text{ with } -1 \le s \le 1$
			\item if $z = \epsilon^{+}(s,t)$ then $J^w_{z}=J_{\text{left},w(s),s,t}$ 	
			\item if $z = \delta^{\pm}(s,t)$ then $J^w_{z} = J^{\pm}_{\text{right},w(s),s,t}$
		\end{enumerate}

		\begin{figure}
			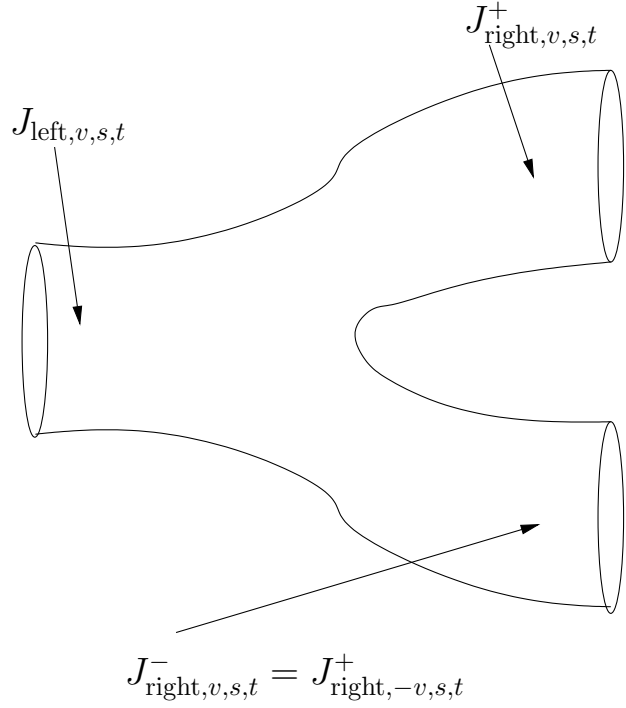
			\caption{$J$ for the symplectic square.}
			\label{fig:jpoppic}
		\end{figure}

	For a fixed complex structure $j$ on $S$, given a family of almost complex structures $J^w_z$ as above, $x$ being a $1$-periodic Hamiltonian loop for $H$ and $y$ a $2$-periodic Hamiltonian loop for $H$, consider the pair-of-pants equation:

\begin{equation}
\label{equation:eqPOP}
              \begin{cases} \begin{array}{ll}
			u : S \rightarrow M \\

			(du - X_H \otimes \beta + Y) \circ j = J^w_{z}(du - X_H \otimes \beta + Y) \\

			{\displaystyle \lim_{s \rightarrow -\infty}} u(\epsilon^{+}(s,t)) =               \begin{cases} \begin{array}{ll}
			y(t) \text{ if } w(-\infty) = v^{i,+}\\
			y(t+1) \text{ if } w(-\infty) = v^{i,-}\\
                \end{array} \end{cases}\\

			{\displaystyle \lim_{s \rightarrow +\infty}} u(\delta^{\pm}(s,t)) = x(t) \\
                \end{array} \end{cases}
\end{equation}

where $X_H$ is the Hamiltonian vector field of $H$, and $\beta$ is a $1$-form on $S$ with $d \beta = 0$ such that $\beta = dt$ near the cylindrical ends of $S$. The holomorphicity equation in \eqref{equation:eqPOP} is perturbed by a small Hamiltonian perturbation $Y$, e.g. with support restricted to $\delta^{\pm}((1,2) \times \mathbb{R}/\mathbb{Z})$, where $Y$ is invariant under $\gamma$, to ensure that there are no solutions that are constant at Hamiltonian loops of $H$.

The moduli space $\mathcal{M}_{prod}^{i,\sigma}(y,x)$ consists of rigid pairs $(w,u)$ with $w : \mathbb{R} \rightarrow S^{\infty}$ a negative gradient flowline of $g$ such that $w(-\infty) = v^{i,\sigma}$ and $w(\infty) = v^{0,+}$, and $u$ being a solution to equation (\ref{equation:eqPOP}).

	\begin{defn}[Symplectic square]
	\label{defn:eqpop}

$$P \mathcal{S}^{i,\sigma} : CH^*(H) \rightarrow (CF^{\bullet} (2 \cdot H))^{2*-i}, \quad P \mathcal{S}^{i, \sigma}(x) = \sum_{y \in \mathcal{L}} \# \mathcal{M}_{prod}^{i,\sigma}(y,x)y$$ where $\#$ counts isolated solutions.

		$$P \mathcal{S}(x) := \sum_{i \ge 0} (P \mathcal{S}^{i,+}(x) + P \mathcal{S}^{i,-}(x)) h^i,$$ which descends to a map on cohomology as shown in Section \ref{subsubsec:eqpopischainmap}. 

	\end{defn}

\subsubsection{The symplectic square is a map on cohomology}
	\label{subsubsec:eqpopischainmap}
		Recall that given a chain complex $(C,d)$ with a chain involution $\iota$, the $\mathbb{Z}/2$-equivariant cohomology of $(C,d, \iota)$ is $$H^*_{\mathbb{Z}/2}(C) := H^*(C_{\mathbb{Z}/2}, d_{\mathbb{Z}/2}),$$ where $$C_{\mathbb{Z}/2} = C[[h]], \ d_{\mathbb{Z}/2} = d + (id + \iota) h.$$ In particular, for $C = CF^*(H) \otimes CF^*(H)$ with $\iota$ defined on generators by $\iota(x \otimes y) = y \otimes x$, we obtain $d_{\mathbb{Z}/2}(x \otimes y) = d(x \otimes y) + (x \otimes y + y \otimes x)h$, where $d$ is the Floer differential. In \cite{seidel}, Seidel constructs a map, the equivariant pair-of-pants product, defining: $$\mathcal{P}_{i,\sigma} : (CF^*(H) \otimes CF^*(H))[[h]] \rightarrow CF^*_{eq}(2 \cdot H),$$ the definition of which is identical to that of $P \mathcal{S}$ in Definition \ref{defn:eqpop} except that to calculate the coefficient of $y$ in $\mathcal{P}_{i,\sigma}(x_+(t) \otimes x_{-}(t))$, we replace $x(t)$ by $x_{\pm}(t)$ in the last line of Equation \eqref{equation:eqPOP}. Seidel shows in \cite[Equation (4.100)]{seidel} that $\mathcal{P}$ satisfies:
\begin{equation} \label{equation:eqpopchainmap} \sum_{j+k=i} d^j_{eq} \circ \mathcal{P}_k(x \otimes y) = \mathcal{P}_i(d(x \otimes y)) + \mathcal{P}_{i-1}(x \otimes y + y \otimes x), \end{equation} where $d$ is the standard Floer differential, $d^i_{eq} = d^{i,+}_{eq} + d^{i,-}_{eq}$, $\mathcal{P}_i = \mathcal{P}_{i,+}+ \mathcal{P}_{i,-}$ and $\mathcal{P} = \sum_{i \ge 0} h^i \mathcal{P}_i$ (this uses the compactification in Section \ref{subsec:equivariantgluing}). Hence $\mathcal{P}$ is a chain map. One can also define the doubling map $$\eta: CF^*(H) \rightarrow (CF^*(H) \otimes CF^*(H))[[h]], \quad \alpha \mapsto \alpha \otimes \alpha,$$ and it is immediate that $P \mathcal{S} = \mathcal{P} \circ \eta$. It is also immediate that $\eta$ descends to a map $\eta: HF^*(H) \rightarrow H^{2*}_{\mathbb{Z}/2}(CF^*(H) \otimes CF^*(H))$ using the involution that swaps the factors on the right hand side. Hence $P \mathcal{S}$ is well defined on cohomology.

\begin{rmk}
The operation $\mathcal{P}$ will contain more information than $P \mathcal{S}$ exactly in the instances that $\mathcal{P}$ does not vanish on the complement of $\text{Im}(\eta)$.
\end{rmk}

\begin{rmk}
\label{rmk:shiftandnegativeaction}
	If $w : \mathbb{R} \rightarrow S^{\infty}$ is a negative gradient flowline for $g$, with $w(-\infty) = v^{i, \sigma'}$ and $w(\infty) = v^{j,\sigma}$ with $j > 0$ and $\sigma, \sigma' \in \{ \pm \}$, then $w$ corresponds to $\sigma \tau^j w'$, where $w'(-\infty) = v^{i-j,\sigma \sigma'}$ and $w'(\infty) = v^{0,+}$. It is important that the $J^w$ satisfy the conditions at the beginning of Section \ref{subsec:eqpop}, because they mean that:
	
\begin{enumerate}
	\item $J^{\tau w}_z = J^w_z$,
	\item $J^{-w}_z = J^w_{\gamma z}$.
\end{enumerate}

	This is used when compactifying the moduli spaces, which is discussed in Section \ref{subsec:equivariantgluing}: given any negative gradient flowline $w$ of $g$ as above, $(w,u)$ satisfies the equivariant Floer equation \eqref{equation:eqPOP} exactly when 
$$ \begin{cases} \begin{array}{ll}
( \tau^j w, u) \text{ does, if } \sigma = + \\
( - \tau^j w, u \circ \gamma) \text{ does, if } \sigma = - \\
\end{array} \end{cases}\\$$

It is necessary here that $\gamma$ is holomorphic, as repeated at the beginning of Section \ref{sec:symplcartan}, so that $u \circ \gamma$ is holomorphic.

\end{rmk}

		\subsection{Quantum cohomology}

	\label{subsec:quantcupprod2}
		For more details on the quantum cup product, see \cite[Chapter 9.2]{jhols} whose exposition we follow. Throughout, $PD$ refers to the Poincar\'e duality operation over $\mathbb{Z}/2$ coefficients. In the case of an open manifold $X$, Poincar\'e duality states that $PD: H^*(X) \cong H^{\text{lf}}_{\dim X - *}(X)$. The open manifolds we will consider have a conical end, so they deformation retract along the fibres of the conical end to a manifold with boundary $(C,\partial C)$ of the same dimension. We interchangeably use $PD: H^*(C) \cong H_{\dim X - *}(C, \partial C)$. The notation $A^{\vee}$ denotes the intersection dual of $A$ (after fixing a basis) with respect to either the pairing $H_*(M) \otimes H^{\text{lf}}_*(M) \rightarrow \mathbb{Z}/2$ (for an open manifold $M$) or $H_*(C) \otimes H_*(C, \partial C) \rightarrow \mathbb{Z}/2$ (for a compact manifold with boundary $C$).

		\begin{defn}
		\label{defn:weakmonotonicity}
			A symplectic manifold $M$ is \textit{weakly monotone} if for every spherical $A \in H_2(M,\mathbb{Z})$ (i.e. $A$ in the image of the Hurewicz homomorphism) such that $\omega(A) > 0$ and $c_1(A) \ge 3-m$, then $c_1(A) \ge 0$, where $\dim(M) = 2m$. 
		\end{defn}

		Let $(M, \omega)$ be a weakly monotone symplectic manifold of dimension $n$, with a fixed almost complex structure $J$ compatible with $\omega$. As an abelian group, $QH^{*}(M) = H^{*}(M) \otimes \Lambda$, where $\Lambda$ is the Novikov ring associated to $\omega$ as in \cite[Chapter 9.2]{jhols}. Specifically, if $\Gamma$ is the image of the Hurewicz homomorphism then $\omega: \Gamma \rightarrow \mathbb{Z}$ is a homomorphism, and $$\Lambda = \left\{ \lambda = \sum_{A \in \Gamma} \lambda_A \cdot q^{A} \biggr\rvert \begin{array}{l} \lambda_A \in \mathbb{Z}/2,\text{ such that for all } c>0 \text{ there are } \\ \text{only finitely many } \lambda_A \neq 0 \text{ with } \omega(A) \le c \end{array} \right\}.$$ There is a natural grading given by $|q^{A}| = 2c_1(A)$. 

		As an important note, we denote the quantum cochains $$QC^*(M) := C^*(M) \otimes_{\mathbb{Z}/2} \Lambda.$$ Then $QH^*(M) = H^*(QC^*(M), d \otimes id_{\Lambda})$, where $d$ is the differential on $C^*(M)$. The quantum product below will be defined at the chain level, and then descend to maps on (co)homology. The definition is as in \cite[Section 11.1]{jholssympl}.

		\begin{defn}[Quantum Product]
			For $a,b \in H^*(M)$, $$a * b := \sum_{A \in H_2(M)} (a*b)_A 	q^A$$ where $(a*b)_A$ is characterized by $\int_Z (a*b)_A \cup c = GW^M_{A,3}(a,b,c)$ for all $c \in H^{\text{lf}}_{|a| + |b|-2c_1(A)}(M)$, where $GW^M_{A,3}$ are the $3$-pointed genus $0$ Gromov-Witten invariants, using notation as in \cite[Chapter 7]{jholssympl}. This descends to a well defined map on homology, $H^*(M) \otimes H^*(M) \rightarrow QH^*(M)$. Extending $\Lambda$-bilinearly defines $*$ on $QH^*(M)$.
		\end{defn}

		Notice that $*$ is compatible with the grading, using $|q^A|=2 c_1(A)$. If $A=0$, the intersection is a point and this recovers the classical intersection product. We will not recap the technical issues of bubbling, multiply covered curves and so on that are covered in \cite[Sections 3-6]{jhols}. We interchangeably use Morse theory for $(C, \partial C)$, which is elaborated on in the next section.

		\subsection{The quantum Steenrod square}
		\label{subsec:SqQviaMorse2}
We will follow the author's previous work in \cite{me}, and give two definitions of the quantum Steenrod square: the first will use Morse theory and the second will involve intersections of chains. These are equivalent, as shown in the cited paper. 

Henceforth we will consider symplectic manifolds $(M,\omega)$ that are monotone or exact (however, see Remark \ref{rmk:generalwm} for a note about a definition in the case of a general weakly monotone symplectic manifolds).
\begin{defn}
\label{defn:monotone}
A symplectic manifold $(M,\omega)$ is \textit{monotone} if there exists a constant $\lambda > 0$ such that every map $u: S^{2} \rightarrow M$ satisfies $$c_{1} (u_{*}([S^{2}])) = \lambda \cdot E(u)$$ where $E(u) = \int_{S^{2}} u^{*} \omega \geq 0$ is the symplectic energy of $u$, and $c_1 = c_1(TM)$.
\end{defn}

\begin{defn}
\label{defn:exact}
A symplectic manifold $(M, \omega)$ is \textit{exact} if the $2$-form $\omega$ is exact.
\end{defn}

Further, we will require our manifold be open and convex at infinity. All of the results in this paper extend to the case where $M$ is closed. The convex at infinity condition, which we will shorten to {\it convex}, implies among other things that $M$ satisfies the conditions of the previous section. Such symplectic manifolds, specifically their symplectic cohomology, were first studied by Ritter in \cite[Section 3]{ritterale}, where the technical machinery is proven in depth. The definition we use is from \cite[Section 3.1]{ritterFTNLB}.

\begin{defn}
\label{defn:convexity}
A symplectic manifold $(M,\omega)$ is {\it convex} if \begin{equation} \label{equation:convex} M = C \cup_{\partial C} (\partial C \times [1,\infty)), \end{equation} for some compact symplectic manifold $C$ with contact-type boundary $(\partial C, \theta)$, such that $(M - C, \omega|_{M - C}) \cong (\partial C \times [1,\infty), d(e^r \theta))$, where $r \in [0,\infty)$ denotes the radial coordinate.
\end{defn}

Examples of such symplectic manifolds are (the completion of) Liouville domains. Note that one can think of a closed symplectic manifold to be defined as in \eqref{equation:convex}, but with $\partial C = \emptyset$.

Recall that $S^{\infty}$ has subsets $S^{i}$ such that each $$S^i = \{ (x_0, x_1, \ldots, x_i, 0, 0, \ldots) \in \mathbb{R}^{\infty} : \sum_i x_i^2 = 1 \},$$ is an antipodally invariant equator, and $\bigcup_{i \ge 0} S^i = S^{\infty}$. We fix a generic almost complex structure $J$ on $M$ compatible with $\omega$, and then perturb this to a choice of almost complex structure $J_{v,z}$ on $M$ for each $v \in S^{\infty}$ and $z \in S^2$, satisfying $J_{v,z} = J_{v, z/(z-1)}$. For brevity if a pair $v \in S^{\infty}$ and $u: S^2 \rightarrow M$ satisfies \begin{equation} du|_z \circ j|_z = J_{v,z}|_{u(z)} \circ du|_z \end{equation} for all $z \in S^2$, then we say that $u$ is $J_{v,z}$-holomorphic.

We will use Morse theory for $(C, \partial C)$, as detailed in \cite{morcombound}. Choosing a Morse function $f$ such that $\nabla f \pitchfork \partial C$ and $f$ is both minimised and constant on $\partial C$, we obtain a Morse complex whose homology $HM_{*}(C,\partial C)$ recovers $H_{*}(C, \partial C)$ in the same way as Morse homology for closed manifolds. Suppose that we pick a choice of perturbations $f_{v,s}$ for $v \in S^{\infty}$ and $s \in [0,\infty)$ such that:
\begin{enumerate}
\item $f_{v,s} = f$ for $s \ge 1$.
\item $f_{v,0}$ are generic, to make all of the moduli spaces transverse (see Remark \ref{rmk:transversesq}).
\item $f_{\tau v, s} = f_{v,s}$ for all $v,s$.
\end{enumerate}
We ensure that $f_{v,s}$ are $C^2$-close to $f$ for all $v,s$: as the Morse-Smale condition is open and dense in the space consisting of pairs of a smooth function with a metrics, we can assume the $f_{v,s}$ are all Morse-Smale. We fix that $f_{v,s}|_{\partial C}$ is independent of $v,s$. Then there are two equivalent definitions of the Morse quantum Steenrod square, depending on whether we use singular or Morse homology on our parameter space $\mathbb{RP}^{\infty} = S^{\infty} / (\mathbb{Z}/2)$. Specifically, recall $g: S^{\infty} \rightarrow \mathbb{R}$ in Equation \eqref{equation:sinftymorse}, and let $\mathcal{Y}$ be $S^2$ with three half-lines attached. There is one incoming edge $e_1$ parametrized by $(-\infty,0]$ and two outgoing edges $e_2,e_3$ parametrized by $[0,\infty)$ The edges are attached at $0,1,\infty$ respectively. Let:
\begin{itemize}
\item $\mathcal{M}_{\text{Sing}, i, A}(x,y)$ consists of pairs $(v, u)$ where $v \in S^{i}$ and $u: \mathcal{Y} \rightarrow M$ such that $\dfrac{d}{dt}(u|_{e_1})(s) = - \nabla f(u|_{e_1}(s)), \ \dfrac{d}{dt}(u|_{e_2})(s) = - \nabla f_{v,s}(u|_{e_2}(s))$ and $\dfrac{d}{dt}(u|_{e_3})(s) = - \nabla f_{-v,s}(u|_{e_3}(s))$. The maps $u|_{S^2}$ are $J_{v,z}$-holomorphic and represent the homology class $A \in H_2(M)$. Asymptotic conditions are given in point $(2)$ below.
\item $\mathcal{M}_{\text{Morse}, i ,A}(x,y)$ consists of pairs $(w, u)$ where $w: \mathbb{R} \rightarrow S^{i}$, $\dot{w} = - \nabla g$ and $u: \mathcal{Y} \rightarrow M$ such that $\dfrac{d}{dt}(u|_{e_1})(s) = - \nabla f(u|_{e_1}(s)), \ \dfrac{d}{dt}(u|_{e_2})(s) = - \nabla f_{w(s),s}(u|_{e_2}(s))$ and $\dfrac{d}{dt}(u|_{e_3})(s) = - \nabla f_{-w(s),s}(u|_{e_3}(s))$. The maps $u|_{S^2}$ are $J_{w(0),z}$-holomorphic and represent the homology class $A \in H_2(M)$. Asymptotic conditions are given in points $(1)$ and $(2)$ below.
\end{itemize}
\begin{enumerate}
\item $w(-\infty) = v^{i, \sigma_1}, \ w(\infty) = v^{0, \sigma_2}$ for $\sigma_1, \sigma_2 \in \{ \pm \}$. 
\item $ u|_{e_i}(s) \rightarrow \begin{cases} \begin{array}{ll}
y \ \text{ for } i=1 \text{ and } s \rightarrow -\infty \\
x \ \text{ for } i=2,3 \text{ and } s \rightarrow \infty \\
\end{array} \end{cases}$
\end{enumerate}

For a generic choice of $J_{v,z}$ and $f_{v,s}$, the moduli spaces above will be smooth manifolds (see Remark \ref{rmk:transversesq}).

\begin{defn}[The Morse Quantum Steenrod Square]
\label{defn:mqssv2}
For $i \in \mathbb{Z}$, $A \in H_2(M)$ and $x,y \in \text{crit}(f)$ such that $|y| = 2|x| - i - 2 c_1(A)$, the coefficient of $y \cdot h^i \cdot q^A$ in $Q \mathcal{S}(x) \in QH^*_{eq}(M)$ is $\# \mathcal{M}_{\bullet, i, A}(x,y) / (\mathbb{Z}/2)$, where $\bullet = \text{Morse}$ or $\text{Sing}$ and the $\mathbb{Z}/2$ action is $-1 \times r^*$, where $r^* u = u \circ r$, and $r: \mathcal{Y} \rightarrow \mathcal{Y}$ acts by fixing $e_1$, swapping $e_2$ and $e_3$ (without changing the parametrisation) and acting by $z \mapsto 1/z$ on $S^2$.
\end{defn}

The fact that these two definitions are equivalent follows from the isomorphism between singular and Morse cohomology, as in \cite{schwarzmorsesingiso}. Note that as we are always working in some finite submanifold $S^i \subset S^{\infty}$, and the $- \nabla g$-flowlines preserve these submanifolds, so the resulting equivalence of the definitions is a straightforward application of this isomorphism (i.e. we do not need to take into account problems involving the infinite-dimensional $S^{\infty}$).

Further, the proof that $Q \mathcal{S}$ is a chain map is proved identically to Section \ref{subsubsec:eqpopischainmap}, except that one replaces $\mathcal{P}$ by $\mathcal{Q}$ and the left hand side of Equation \eqref{equation:eqpopchainmap} with $d \circ \mathcal{Q}_i(x \otimes y)$ (which corresponds to the fact that, unlike equivariant Hamiltonian Floer cohomlology, the equivariant quantum cohomology uses the trivial $\mathbb{Z}/2$-action).

We give in Definition \ref{defn:eqpopqss} an alternative description of the quantum Steenrod square involving pseudocycles. Using the pseudocycles arising from Morse theory (see \cite{schwarzmorsesingiso}), one can show that Definition \ref{defn:mqssv2} is a particular case of Definition \ref{defn:eqpopqss}. The following definition of the quantum Steenrod square will be used strictly for computational purposes.

Recall that for a manifold with boundary $(C, \partial C)$, a homology class $\alpha \in H_*(C, \partial C)$ is realizable if there is a manifold with boundary $A$ and $\mu: (A, \partial A) \rightarrow (C, \partial C)$ with $\mu_*([A]) = \alpha$. Thom proved in \cite{thom} (for $\partial C = \emptyset$) that all homology classes over $\mathbb{Z}/2$ coefficients are realizable. A slight modification of the constructive proof by Buoncristiano-Hacon in \cite[Theorem B]{buonhacon} yields that all elements of $H_*(C,\partial C)$ are realisable. Hence, we associate a homology class $\alpha$ interchangeably with a pair $(A,\mu)$. The conditions we placed on $M$ allow us to pick representatives of $x \in H^{\text{lf}}_*(M)$ to be contained in $H_*(C, \partial C)$.

Fix a basis $\{ b \} \in \mathcal{B}$ of $H^*(M)$. Denote by $\beta^{\vee}: Y_b \rightarrow M$ a pseudocycle representative of $PD(b^{\vee})$, and by $\beta^{\vee} \times \beta^{\vee}: Y_b \rightarrow M \times M$ the map such that $\beta^{\vee} \times \beta^{\vee}(y) = (\beta^{\vee}(y), \beta^{\vee}(y))$. Observe that there is an involution $$\iota = id_M \times (-id_{S^{\infty}}) : M \times S^{\infty} \rightarrow M \times S^{\infty}.$$ We fix $x \in H^*(M)$. To calculate $Q \mathcal{S}(x)$, we begin by choosing a sequence of pairs $(\chi, \mu_i: \chi_i \rightarrow M \times S^i)_{i=0}^{\infty}$ where $\chi$ is a smooth manifold, $\chi_i = \chi \times S^i$ and $\mu_i$ is a smooth map. The $\mu_i: \chi_i \rightarrow M \times S^{i}$ satisfy:

		\begin{enumerate}
		\item For $\pi_{2}: M \times S^i \rightarrow S^i$ the second projection, $\pi_{2}(\mu_i(x,v)) = v$ for all $(x,v) \in \chi_i$.
		\item The restriction $\mu_i |_{\chi_j} = \mu_j$ for $j \le i$.
		\item For $\pi_{1}: M \times S^i \rightarrow M$ the first projection, for any $v \in S^{i}$ then \begin{equation} \label{equation:alphav} \mu_v:= \pi_1 \circ \mu|_{\chi \times \{ v \}} : \chi_v := \chi \times \{ v \} \rightarrow M \end{equation} is a pseudocycle representative of $A$ in $M$ (and is well defined by (2) above).
		\item For $b \in \mathcal{B}$, the intersection \begin{equation} \label{tripleintersection}  (\beta^{\vee} \times \beta^{\vee})(Y_b) \times S^{i}) \cap \euscr{W}(\chi_i) \end{equation} is transverse in $M \times M \times S^{i}$, where $\euscr{W}: \chi \times \chi \times S^i \rightarrow (M \times M) \times S^i$ is defined by $(x, x' ,v) \mapsto (\mu_i(x,v), \mu_i(x',-v), v)$.
		\end{enumerate}

The previous construction using Morse flowlines in fact demonstrates that it is possible to choose such pseudocycles. Fixing $i \ge 0, \ A \in \Gamma, \ x \in H^*(C),$ and $Y \in H_*(C, \partial C)$, we choose $X_v := \mu_v(v, \chi)$ as above, and define a moduli space:

\begin{defn}
\label{defn:eqpopqss}
$$\mathcal{M}_{i,A}(x,Y) = \left\{ (v,u) \Biggr\rvert \begin{array}{l} v \in S^i, \ u: S^2 \rightarrow C, \ u_*[S^2] = A, \\ u \text{ is } J_{v,z} \text{-holomorphic}, \ u(0) \in Y, \\ u(1) \in X_v, \ u(\infty) \in X_{-v} \end{array} \right\}.$$
\end{defn}

For generic choices of data, if $|Y| = 2|x| - i - 2c_1(A)$ then $\mathcal{M}_{i,A}(x,Y)$ is a $0$-dimensional manifold. Define $Q \mathcal{S}_{i,A}(x) \in H^{2|x|-i-2c_1(A)}(C)$ by $\int_Y Q \mathcal{S}_{i,A}(x) = \# \mathcal{M}_{i,A}(x,Y)$ for all $Y \in H_{2|x|-i-2c_1(A)}(C, \partial C)$.

\begin{defn}[The Quantum Steenrod Square]
$$Q \mathcal{S}: H^*(C) \rightarrow QH^{2*}(C)[[h]], \ Q \mathcal{S}(x) = \sum_{i \ge 0, \ A \in \Gamma} Q \mathcal{S}_{i,A}(x) h^i q^A.$$

We extend this over $\Lambda$ by requiring $$Q \mathcal{S}\left(\sum_{A \in \Gamma} x_A q^A \right) = \sum_{A \in \Gamma} Q \mathcal{S} (x_A) q^{2A}.$$
\end{defn}

For $a \in H^*(M)$,
		\begin{equation} \label{equation:somepropsforqs} Q\mathcal{S}_{i,0}(a) = Sq^{|a|-i}(a) \text{ and } \sum_{A \in H_2(M)} Q\mathcal{S}_{0,A}(a) q^A = a * a. \end{equation}

\begin{rmk}[Transversality of moduli spaces]
\label{rmk:transversesq}

Suppose that the space $\mathcal{M}_i(A,J_{v,z})$ of pairs $(v,u)$ is defined by $v \in S^i$, the map $u: S^2 \rightarrow M$ is simple and $J_{v,z}$-holomorphic, and $u_*[S^2] = A$. Assume for now that $\mathcal{M}_i(A,J_{v,z})$ is a smooth manifold. Recall that there are evaluation maps $ev_z$, evaluating $u$ at $z$, and a forgetful map $\pi_{S^{\infty}}$ mapping to the $v$ factor. Then the evaluation map $ev_0 \times ev_1 \times ev_{\infty} \times \pi_{S^{\infty}}: \mathcal{M}_i(A,J_{v,z}) \rightarrow M \times M \times M \times S^{i}$ descends to a map on the quotient by $\mathbb{Z}/2$ (which is induced by the composition with the  map $z \mapsto z/(z-1)$ on $\mathcal{M}_i(A,J_{v,z})$) and further defines a pseudocycle, as in the classical case in \cite[Theorem 6.6.1]{jholssympl}. Suppose that  $a,b \in \text{crit} (f)$ are critical points of a fixed Morse function. For the coefficient of $b h^i$ in $Q \mathcal{S}(a)$, we count the intersection number of $ev_0 \times ev_1 \times ev_{\infty} \times \pi_{S^{\infty}}$ with another pseudocycle landing in $M \times ((M \times M) \times_{\mathbb{Z}/2} S^{i})$, which is the quotient by $\mathbb{Z}/2$ of the following $\mathbb{Z}/2$-equivariant pseudocycle: 
$$\euscr{W}^{i}(a, b) :  \overline{W^u(b,f)} \times \overline{W^s(a,f)} \times \overline{W^s(a,f)} \times S^i \rightarrow M \times ((M \times M) \times S^{i}),$$ where $\overline{W^u(b,f)}$ is the compactification of the unstable manifold of $b$ for the Morse function $f$ (similarly $W^s$ denotes the stable manifold). Details of this construction are in \cite[Lemma 4.5]{schwarzmorsesingiso}. There is further an evaluation map $E: \overline{W^u(b,f)} \rightarrow M$, such that $E$ is a pseudocycle. Let $\phi_{v,t}$ be the $1$-parameter family of diffeomorphisms defined for $v \in S^i$ and $t \ge 0$ by $$ \dfrac{d \phi_{v,t}}{dt}(s) = - \nabla f^2_{v,s} \text{ and } \phi_{v,0} = id. $$Then $$\euscr{W}^{i}(a, b)(y,x,x',v) := (E(y), \phi_{v,1}^{-1} E(x), \phi_{-v,1}^{-1} E(x'), v),$$ where the evaluation maps are all abusively denoted $E$. This encodes the $\mathbb{Z}/2$-equivariant incidence data associated to the Morse flowlines of $f_{v,s}$ (alternatively, the directly defined pseudocycles $\mu_i$ can be used in place of the evaluation maps). This is then a classical transversality problem, requiring us to make a generic choice of $f_{\cdot,0}: S^i \times M \rightarrow \mathbb{R}$, to ensure that our moduli spaces are cut out transversely as we required. It is immediate that this is an unconstrained transversality problem (i.e. not a $\mathbb{Z}/2$-equivariant one) because we observe that $\mathbb{Z}/2$-equivariant transversality for $v \in S^i$ is achieved when we ascertain nonequivariant transversality for $v \in D^{i,+}$, the upper $i$-dimensional hemisphere.

Hence it remains to determine whether the space $\mathcal{M}_i(A,J_{v,z})$ is cut out transversely. One follows the proof of \cite[Proposition 6.7.7]{jholssympl}, replacing $\mathcal{J}^l = \{ \{ J_{z} \} \}$ by $\mathcal{J}^l_i = \{ \{ J_{v,z} \} \}$ (these being Banach spaces of almost complex structures on $M$ of class $C^l$). Observe that the universal moduli space also has an extra factor that is $S^i$. This, being a finite dimensional manifold, does not affect the analytic outcome of the proof, hence one can find such a $J_{v,z}$ for each $i \in \mathbb{Z}_{\ge 0}$. However, in general we wish to choose a single $J_{v,z}$ ($v \in S^{\infty}$) that is sufficient for all $i$ simultaneously. To do this, observe that $\mathcal{J}^l_{i+1} \subset \mathcal{J}^l_{i}$ for each $i$. In fact given any $\{ J_{v,z} \} \in \mathcal{J}^l_{i}$, we may extend piecewise $C^l$ to some $\{ J_{v,z} \}$ for $v \in S^{i+1}$ (we pick a $C^l$ nullhomotopy of $\{ J_{v,z} \}$ in $\mathcal{J}^l$, and use this to extend to a $\{ J_{v,z} \}$ that is $C^l$ everywhere except perhaps along the ``equator" $S^i \subset S^{i+1}$). Using a $C^l$ approximation of this, we see that $\mathcal{J}^l_{i+1} \subset \mathcal{J}^l_{i}$ is in fact dense. Hence, using the Banach property of each $J^l_i$, we observe that the intersection of these nested sets is also dense in each $\mathcal{J}^l_{i}$, as required.
\end{rmk}

\begin{rmk}
	In Definition \ref{defn:mqssv2}, we could use either of the two moduli spaces listed because these correspond to singular and Morse cohomology of $\mathbb{RP}^{\infty}$ respectively. Using an appropriate Morse function $g$, these two definitions are equivalent in the same way as for finite dimensional manifolds, and we will use them interchangeably.
\end{rmk}

\begin{rmk}
\label{rmk:generalwm}

For general weakly monotone symplectic manifolds, there may be problems arising from multiply covered curves (this is avoided in the monotone case because such curves arise in families of codimension at least $2$, and in the exact case because there are no holomorphic curves). In light of this, one would need to add an inhomogeneous term to the equation $du \circ j = J \circ du$, resulting in something of the form $$(du - Y_{v,z}) \circ j = J \circ (du - Y_{v,z}),$$ where $v \in S^{\infty}$, $z \in S^2$ and $Y_{v,z} = Y_{-v, z/(z-1)}$. 
\end{rmk}

		\subsection{The PSS isomorphism}
		\label{subsec:pssisomprelim}

Recall that for some $H: M \rightarrow \mathbb{R}$ where $H$ is $C^2$-small and Morse, there is an isomorphism of rings due to Piunikhin, Salamon and Schwarz \cite[Example 3.3]{PSS}, \begin{equation} \label{equ:PSI} \Psi: QH^{*}(M) \rightarrow HF^*(H). \end{equation} 

We pick an almost complex structure $J$ on $M$ and an interpolation $H_{s}$ for $s \in \mathbb{R}$, such that $H_s = H$ for $s$ near $-\infty$ and $H_s = 0$ for $s \ge 0$. Consider $\mathbb{C}$ as a disc with a cylindrical end: we begin with $\mathbb{C}^*$ parametrised as the cylinder with logarithmic coordinates, $$(s,t) \in \mathbb{R} \times S^1 \rightarrow e^{-2\pi(s +  i t) } \in \mathbb{C}^*.$$ 

Suppose $u: \mathbb{C}^* \rightarrow M$ satisfies \eqref{equation:perturbedJholoc} (using logarithmic coordinates). We define the {\textit geometric energy} of $u$ to be $E(u) = \int_{\mathbb{C}^*} |\partial_S u|^2_J ds ^ dt,$ where $| \cdot |_J$ is defined using the metric $\omega(-,J-)$ induced by the symplectic form $\omega$ of $M$ and the compatible almost-complex structure $J$. Using removal of singularities, $u$ is finite energy exactly when it extends to a smooth map $u : \mathbb{C} \rightarrow M$ satisfying a perturbed $J$-holomorphic equation \eqref{equation:perturbedJholoc}, such that ${\displaystyle \lim_{s \rightarrow - \infty}} u(s,t)$ is a Hamiltonian loop, and $u$ is $J$-holomorphic on the unit disc.

		Let $$\mathcal{L}_1 = \{ y: \mathbb{R} / \mathbb{Z} \rightarrow M | \dot{y} = X_H(y) \}.$$ If $y \in \mathcal{L}_1$ then the coefficient of $y$ in $\Psi(x)$ is the number of ``spiked discs" asymptoting to $y$ and intersecting $x$. These spiked discs are smooth maps $$u: \mathbb{C} \rightarrow M$$ where we parametrise $\mathbb{C}$ as above, satisfying \begin{equation} \label{equation:perturbedJholoc} \partial_s u + J (\partial_t u - X_{H_s} )= 0 \end{equation} and $u$ is $J$-holomorphic on the unit disc, with ${\displaystyle \lim_{s \rightarrow -\infty}} u(s,t) = y(t)$, and $u(0) \in PD(x)$: the notation means that we have fixed a generic chain representative $X$ of $PD(x)$ and require $u(0) \in X$. See Figure \ref{fig:pss}. Alternatively, having fixed a Morse function $f$ (abusing notation where here we let $x \in \text{crit}(f)$) we require that $u(0) \in W^s(x,f)$, the stable manifold of $x$ using $f$.

		\begin{figure}
			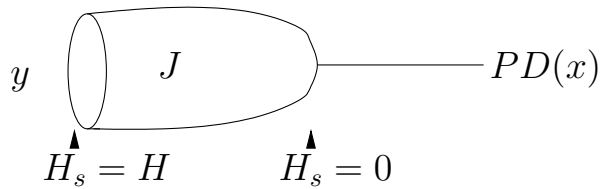
			\caption{Spiked discs.}
			\label{fig:pss}
		\end{figure}
		
		\begin{defn}[The PSS Map]
		\label{defn:PSS}
			The PSS map $\Psi: QC^*(M) \rightarrow CF^*(H,J)$ is defined by $$\Psi(x) = \sum_{y \in \mathcal{L}_1} n_{y,x} \cdot y$$ where $n_{y,x} $ is the count of ``spiked discs" as above, with fixed parametrisation. This map induces a well-defined map on cohomology, and there is an inverse map (counting $J$-holomorphic discs parametrised by $(s,t) \mapsto e^{2\pi(s+ i t)}$).
		\end{defn}
			
		\begin{rmk}
			For Liouville domains, exactness implies $QH^*(M) \cong H^*(M)$. Hence in this case $\Psi: H^*(M) \rightarrow HF^*(H)$.
		\end{rmk}

	\begin{rmk}
		The Hamiltonian boundary condition in the definition of spiked discs can be thought of as an infinite dimensional analogue to the result in Morse theory, that (on a closed manifold), given a Morse flowline $u: \mathbb{R} \rightarrow M$, the limits of $u(s)$ as $s \rightarrow \pm \infty$ are Morse critical points.
	\end{rmk}

		\subsection{The $c^*$-map}
A Hamiltonian $H$ is radial at infinity if on the conical end (i.e. the region $\partial C \times [1,\infty)$), there is some $R$ such that $H(z,r) = \lambda_H \cdot r$, where $(z,r)$ are the coordinates on the conical end, and $r > R$. Let $H, H'$ be Hamiltonians that are radial at infinity with $\lambda_{H} < \lambda_{H'}$. There exist continuation maps $\Phi_{H,H'} : HF^*(H) \rightarrow HF^*(H')$, as in Section \ref{subsec:equivariantcontinuationmaps} (nonequivariant continuation maps in this section correspond to the $(i=0)$-term). We define $$SH^*(M) := \underset{H}{\varinjlim} \ HF^*(H),$$ where the direct limit is taken over Hamiltonians that are radial at infinity. For a $C^2$-small $H_0$, there exists the PSS isomorphism from Definition \ref{defn:PSS}, $\Psi: QH^*(M) \rightarrow HF^*(H_0)$. There is also a natural map $\xi : HF^*(H_0) \rightarrow SH^*(M)$ to the direct limit. This allows the following definition: \begin{defn} $$c^* := \xi \circ \Psi: QH^*(M) \rightarrow SH^*(M).$$ \end{defn}

\subsection{Equivariant compactness, gluing and bubbling}
\label{subsec:equivariantgluing}

We will describe equivariant compactness, and include some words about equivariant gluing. For a more detailed description of the parameter space and the various breakings of Morse flowlines on $S^{\infty}$, refer to Section 3 and Section (4b) after Addendum 4.6 of \cite{seidel} 

We denote by $\euscr{P}^{i,\sigma}$ the space of $w: \mathbb{R} \rightarrow S^{\infty}$ such that $w(-\infty) = v^{i,\sigma}$ and $w(\infty)= v^{0,+}$ (recalling the Morse function $g$ and the critical points $v^{i,\sigma}$ from Section \ref{subsec:equivHF}). These spaces may be compactified to $\overline{\euscr{P}}^{i,\sigma}$ by adding broken Morse trajectories, i.e. tuples $([w_0],...,[w_{j-1}],w_j, [w_{j+1}],...,[w_l])$ such that:
\begin{itemize}
\item $w_p: \mathbb{R} \rightarrow S^{\infty}$ for $p=0,...,l$, such that $\partial w_n / \partial t|_{t'} = - \nabla g(w_n(t'))$
\item $w_p(\infty) = w_{p+1}(-\infty)$ for all $p=0,1,\ldots l-1$.
\item $[w_p] = w_p / \mathbb{R} \in \euscr{Q}^{i,\sigma} := \euscr{P}^{i,\sigma} / \mathbb{R}$ for $p \neq j$ are unparametrised trajectories obtained by quotienting by the translation in $\mathbb{R}$. 
\end{itemize}
In particular, exactly one of the resulting ``pieces" of the broken Morse trajectory is parametrised. Note that it is possible for one of the $w_p$ to be constant at a critical point of $g$, assuming that this is the flowline that is parametrised. Observe also that the $\euscr{Q}^{i,\sigma}$ have similar compactifications $\overline{\euscr{Q}}^{i,\sigma}$ where all of the limiting flowlines are unparametrised. 

In general, given a sequence $[w^k,u^k] \in \mathcal{M}_{\text{eq}}^{i,\sigma}(y,x)$ (in which case we are using unparametrized flowlines on $S^{\infty}$, hence use parameter spaces $\euscr{Q}^{i,\sigma}$), there is a ``forgetful" map $\mathcal{M}_{\text{eq}}^{i,\sigma}(y,x) \rightarrow \euscr{Q}^{i,\sigma}$. One adds limit points in such a way as to respect the compactification. In particular, we add limit points in the following way:

\begin{itemize}
\item some subsequence of the $w^k$ converges to the limit $w^{\infty} = (w_0, \ldots , w_l)$ on compact subsets, as previously detailed.
\item some subsequence of the $u^k$ converges to a broken Floer trajectory $u^{\infty} = (u_0, ..., u_l)$ with respect to the complex structure $J_{v,z}$, with $u_n$ corresponding to (i.e. satisfying Equation \eqref{equation:jweq} with respect to) the $J^{w_n}_{s,t}$. There may also be legitimate (i.e. nonequivariant) Floer trajectories between $u_n$ and $u_{n+1}$ for some $n$. 
\end{itemize}

This defines a compactification $\overline{\mathcal{M}_{\text{eq}}^{i,\sigma}}(y,x)$. One can proceed similarly for a sequence $(w^k,u^k) \in \mathcal{M}^{i,\sigma}_{\text{prod}}(y,x)$, defining the compactification $\overline{\mathcal{M}^{i,\sigma}}_{\text{prod}}(y,x)$ of $\mathcal{M}^{i,\sigma}_{\text{prod}}(y,x)$ in such a way as to respect the compactification $\overline{\euscr{P}^{i,\pm}}$ of $\euscr{P}^{i,\pm}$. Instead of giving a full description, we will note below the cases that we will use. However, it should perhaps be noted that this is a parametrised version of the standard pair-of-pants breaking as in \cite{equivcompactsal}, consisting of a pair-of-pants with some collection of Floer trajectories attached at each infinite cyclindrical end. The parametrised element $w_j$ of $([w_0], \ldots [w_l])$ is then the flowline that is used for defining Equation \eqref{equation:eqPOP} of the central pair-of-pants.

The important situations that we consider in the compactification of the equivariant pair-of-pants case are the following: for the isolated points of $\mathcal{M}^{i,\sigma}_{\text{prod}}(y,x)$, there is a finite count. In the one-dimensional case, the limit points may only occur in the following ways:
\begin{itemize}
\item the sequence of pairs converges in $\mathcal{M}_{\text{prod}}^{i,\sigma}(y,x)$, so no new limit points are added.
\item there is convergence in $\euscr{P}^{i,\sigma}$ of the $w^k$ component (under the forgetful map), so a new Floer trajectory appears.
\item there is convergence in $\overline{\euscr{P}}^{i,\sigma}$, and for degree reasons $l=2$ above and no new Floer trajectories appear.
\end{itemize}
This shows that $\overline{P}^{i,\sigma}$ is a chain map: cf Equation \eqref{equation:eqpopchainmap}. Likewise, the compactification of $\mathcal{M}^{i,\sigma}_{\text{eq}}(y,x)$ yields that $$\sum_{j+k=i} d^j_{eq} \circ d^k_{eq} = 0,$$ where $d^i_{eq} = d^{i,+}_{eq} + d^{i,-}_{eq}$, and hence $d_{eq}$ is a differential.

We will not repeat the standard gluing argument. The equivariant gluing theorem involves two different gluing theorems working together. Consider a ``broken solution", in $\overline{\mathcal{M}_{\text{prod}}^{i,\sigma}}$ (where we have fixed some $x,y$ and omit from the notation). This is a pair consisting of a broken Morse flowline $w^{\infty} \in \overline{P}^{i,\sigma}$, alongside some $u^{\infty}$ such that $[w^{\infty}, u^{\infty}] \in \overline{\mathcal{M}^{i,\sigma}_{\text{prod}}} -\mathcal{M}_{\text{prod}}^{i,\sigma} $. Observe the following: one first defines an approximate glued solution nearby to $w^{\infty}$ in the parameter space $\overline{\euscr{P}}^{i,\sigma}$. Recall that there is a classical gluing theorem that one can apply for the parameter space of Morse flowlines $\overline{\euscr{P}}^{i,\sigma}$ in $S^i$. Then one can define an approximate parametrised pair-of-pants using this new, glued parameter: see the end of Section (4b) in \cite{seidel}, with reference given to the continuation maps in \cite[Lemma 3.12]{salamonfloer}. It remains possible to define a bounded right inverse, and one can still apply the contraction mapping theorem in the usual sense when the gluing parameter is sufficiently large (although the bound on the gluing parameter for the parametrised Floer trajectory will of course depend on a choice of glued flowline in $\overline{\euscr{P}}^{i,\sigma}$, which itself is subject to choosing a sufficiently large gluing parameter). Note that in this case, in general there are two gluing parameters: one assigned to the approximate Morse flowline in $S^i$ and one to the approximate parametrised pair-of-pants. For the particular cases as given above (i.e. the elements of the compactified moduli space as given in the bullet points), we can see that we in fact only obtain a single gluing parameter. In particular, \begin{itemize}
\item in the first case the gluing is unecessary (as we are not considering a broken solution),
\item in the second case one does not need to consider a gluing for the parameter space, because the flowline is not broken, so one obtains an approximate parametrised pair-of-pants. One then uses the contraction mapping theorem applied to the pair consisting of the honest Morse flowline in $S^i$ and the approximate parametrised pair-of-pants to deduce that there is an actual solution nearby. 
\item in the third case there are no broken parametrised pair-of-pants, but there is a broken Morse flowline. Hence one defines an approximate glued flowline in $\overline{\euscr{P}}^{i,\sigma}$. Then one uses the contraction mapping theorem, where the subtlety is that the pair now consists of an approximate glued flowline in $S^{i}$ and an approximate parametrised pair-of-pants, (which is only an approximate solution because the parameter that we are using has changed, and hence the new pair-of-pants may not exactly satisfy the required equation using the approximate Morse flowline).
\end{itemize}
All other cases are of codimension strictly greater than $1$, hence do not appear in moduli spaces of dimension at most $1$. Thus we observe that in all three cases there is exactly one gluing parameter, yielding the one-dimensional family of solutions.

To give a brief note on bubbling, one deals with bubbling in the same way as one does classically. In the exact case there is no problem because there do not exist any $J$-holomorphic spheres. In the monotone case, solutions involving bubbles arise in families of codimension at least $2$. As we only consider zero and one dimensional moduli spaces in our applications (we either count the number of elements of a zero dimensional moduli space to define an operation, or we consider the endpoints of a $1$-dimensional moduli space to show that two operations are the same), for generic choices such bubbling solutions cannot appear.

\section{Equivariant continuation maps}
\label{subsec:equivariantcontinuationmaps}

Recall that for $H \le H'$ being radial at infinity, there exist continuation maps $\Phi_{H,H'}: HF^*(H) \rightarrow HF^*(H')$. See \cite[Section 3.4]{salamonfloer} for the case of a compact symplectic manifold, and \cite[Section 2.9]{ritternov} for the case where $M$ is open and convex at infinity. These satisfy three conditions:

\begin{enumerate}
\item A generic homotopy of the data relative to the endpoints induces a chain homotopy on $CF^*(H)$.
\item $\Phi_{H,H'} = \Phi_{H'',H'} \circ \Phi_{H,H''}$ for any $H \le H'' \le H'$.
\item $\Phi_{H,H} = id$.
\end{enumerate}

These continuation maps are constructed by counting $s$-perturbed Floer trajectories. The $s$-perturbation changes the dimension of the moduli space that we consider (there is no longer an $\mathbb{R}$-action by translating $s$). It is important that $H \le H'$ for these continuation maps to exist in general, although for closed monotone symplectic manifolds there are continuation maps between any two Hamiltonians.

Let $H \le H'$ be radial at infinity. Pick $H_s$ for $s \in (-\infty,\infty)$ and $S_0 \in \mathbb{R}_{>0}$ such that $H_s = H'$ for $s \le -S_0$ and $H_s = H$ for $s \ge S_0$, where there is some $R \ge 1$ such that $\partial_s H_s(r,m) \le 0$ for $r \ge R$ on the conical collar of $M$. Let $i \in \mathbb{Z}_{\ge 0}$, and $x \in CF^*(2 \cdot H)$ and $y \in CF^*(2 \cdot H')$. Define $\mathcal{M}_{\text{cont}}^{i,\sigma}(y,x)$ to be the set of pairs $(w,u)$ where $w$ is a $- \nabla g$ flowline on $S^\infty$ with $w(-\infty) = v^{i,\sigma}$ and $w(\infty) = v^{0,+}$ and $u: \mathbb{R} \times \mathbb{R}/2 \mathbb{Z} \rightarrow M$ satisfying $$ \partial_s u + J_{eq,w(s),t} \partial_t u = - \nabla H_s,$$ with $u(s,t) \rightarrow x(t)$ as $s \rightarrow \infty$ and 

$$u(s,t) \rightarrow \begin{cases} \begin{array}{ll}
				y(t) \text{ if } \sigma = +\\
				y(t+1) \text{ if } \sigma = -\\
                \end{array} \end{cases} \text{ as } s \rightarrow -\infty,$$

\begin{defn}[Equivariant continuation maps]
$$\Phi^{i,\sigma}_{eq, H, H'}: HF^*_{eq}(2 \cdot H) \rightarrow HF^{*-i}_{eq}(2 \cdot H'),$$

$$\Phi^{i,\sigma}_{eq,H, H'}\left(\sum_{j \ge 0} x_j h^j \right) = \sum_{j \ge 0} \sum_{y \in \mathcal{L}} \# \mathcal{M}_{\text{cont}}^{i,\sigma}(x_j,y) \cdot y \cdot h^{j}.$$
Yielding

$$\Phi_{eq,H, H'}: HF_{eq}^{*} (2 \cdot H) \rightarrow HF_{eq}^{*} ( 2\cdot H'), \quad \Phi_{eq,H, H'} = \sum_{i, \sigma} h^i \Phi^{i,\sigma}_{eq,H, H'}.$$
\end{defn}

The fact that this is well defined needs the equivariant compactness and gluing theorem as in Section \ref{subsec:equivariantgluing}, to deduce that $$\sum_{k=i+j, \ \sigma_1 \cdot \sigma_2 = +} \phi^{i,\sigma_1}_{eq,H,H'} \circ d_{eq}^{j, \sigma_2} = \sum_{k=i+j, \ \sigma_1 \cdot \sigma_2 = +} d_{eq}^{i, \sigma_1} \circ \phi^{j,\sigma_2}_{eq,H,H'},$$ for all $k \in \mathbb{Z}_{\ge 0}$, and similarly for $\sigma_1 \cdot \sigma_2 = -$.

\begin{lemma}
\label{lemma:eqcontinmaps}
As with nonequivariant continuation maps:
\begin{enumerate}
\item A generic homotopy of the data, subject to $J^w_{s,t}$ satisfying conditions (1), (2) and (3) in Section \ref{subsec:equivHF}, induces a chain homotopy on $CF^*_{eq}(2 \cdot H)$.
\item $\Phi_{eq,H,H'} = \Phi_{eq,H'',H'} \circ \Phi_{eq,H,H''}$ for any $H \le H'' \le H'$.
\item $\Phi_{eq,H,H} = id$.
\end{enumerate}
\end{lemma}
\begin{proof}
The proof of $(1)$ is the same as in the nonequivariant case: see \cite[Lemma 3.12]{salamonfloer}. For $(2)$ we use equivariant gluing for the right hand side of the equality, as in Section \ref{subsec:equivariantgluing}, using as our data associated to a sufficiently large gluing parameter $\lambda > S_0$ being (for $v \in S^{\infty}$ and $(s,t) \in \mathbb{R} \times \mathbb{R}/\mathbb{Z}$):
\begin{itemize}
\item $J_{v,s,t} = J_{eq,v,t}$.
\item $H_{s,\lambda}= H^1_{s+\lambda}$ for $s \le 0$ and  $H_{s,\lambda}= H^2_{s-\lambda}$ for $s \ge 0$, where $H^1_s$ is the $s$-dependent Hamiltonian used for $ \Phi_{eq,H'',H'}$ and $H^2_s$ is the $s$-dependent Hamiltonian used for $\Phi_{eq,H,H''}$. 
\end{itemize}
This is well defined and $H_{s,\lambda}$ is smooth because $\lambda > S_0$. We then appeal to $(1)$ to allow a different choice of equivariant data on the left hand side of the equality.

To prove $(3)$ we use $(1)$ and note that we may choose $H_s \equiv H$, and then solutions that are not constant at a Hamiltonian loop come in an $\mathbb{R}$-family by translating in $s$, hence are nonisolated.
\end{proof}

\begin{rmk}[Equivariant Auxiliary Data]
\label{rmk:equancdat}
When constructing quantum cohomology, there are auxiliary data in the form of the almost complex structure $J$ and the (perturbed) Morse functions $f^p_s$ for $p=1,2,3$. When constructing Floer cohomology the auxiliary data comprises of an almost complex structure $J$ and the Hamiltonian $H$, along with a small perturbation of $H$. To make these concepts equivariant, we need to introduce a dependency of this data on $v \in S^{\infty}$. However, this dependency may not need to be applied to all of the ancillary data. All that one requires is that ``sufficiently much" of the data is chosen with a $v$-dependence, and the choices we make are generic. 
\end{rmk}

\section{The Equivariant PSS isomorphism}
\label{sec:eqpssisom}

		The PSS isomorphism $\Psi$ of Definition \ref{defn:PSS} extends to an isomorphism $$\Psi_{eq} : QH^*(M)[[h]] \rightarrow HF^{*}_{eq} (2 \cdot H)$$ when $H$ is $C^2$-small. Recall the spiked discs used in the definition of $\Psi$. We will replace these with ``$i$-equivariant spiked discs". We parametrize $\mathbb{C}$ away from $0$ by logarithmic coordinates $e^{-\pi(2s+i t)}$ for $(s,t) \in (-\infty, \infty) \times \mathbb{R}/2 \mathbb{Z}$ as in Section \ref{subsec:pssisomprelim} (note we use $\pi i t$ because our cylinder has circumference $2$). Fixing an almost complex structure $J$ on $M$, for $(v,z) \in S^{\infty} \times \mathbb{C}$, pick almost complex structures $J^{\Psi}_{v,z}$ on $M$ with the conditions that: 
		\begin{enumerate}
			\item $J^{\Psi}_{v,s,t} = J^{\Psi}_{v,s,t+2}$.
			\item $J^{\Psi}_{v,s,t} = J$ for $s \ge 0$. 
			\item $J^{\Psi}_{-v,s,t} = J^{\Psi}_{v,s,t+1}$. 
			\item There is $S_0>0$ such that $J^{\Psi}_{v,s,t} = J_{eq,v,t}$ for $s \le -S_0$.
		\end{enumerate}	

		Given a flowline $w$ of $- \nabla g$ on $S^{\infty}$, we define $J^w_{\Psi,s,t} = J^{\Psi}_{w(s),s,t}$. Pick an interpolation $H_{s}$ for $s \in \mathbb{R}$ such that $H_s = H$ for $s$ near $-\infty$ and $H_s = 0$ for $s \ge 0$. As in the discussion in Remark \ref{rmk:equancdat}, we need not put any $v$ dependence on the $H_s$ assuming that the $J_v$ are generic. 

		\begin{defn}[$i$-equivariant spiked discs]
		Given $x \in \text{crit}(f)$ and $y \in \mathcal{L}$, an $i$-equivariant spiked disc from $x$ to $y$ is a triple $(w,u, \alpha)$ where $w: \mathbb{R} \rightarrow S^{\infty}$ is a $-\nabla g$ flowline from $v^{i,\sigma}$ to $v^{0,+}$ on $S^{\infty}$, the map $$u: \mathbb{C} \rightarrow M$$ is smooth, and $\alpha: [0,\infty) \rightarrow M$ is smooth, with:

		\begin{itemize}
\item ${\displaystyle \lim_{s \rightarrow -\infty}} u(s,t) = \begin{cases} \begin{array}{ll}
 		y(t) \text{ if } \sigma = + \\
		y(t+1) \text{ if } \sigma = - \\
                \end{array} \end{cases}$ 
\item $\partial_s u + J^w_{\Psi,s,t} (\partial_t u - X_{H_s}) = 0 \text{ on } \mathbb{C},$
\item $\dot{\alpha}(t) = - \nabla f(\alpha(t))$ and $\alpha(\infty) = x$.
\item $u(0) = \alpha(0)$.
		\end{itemize}		
		\end{defn}

		\begin{defn}
		\label{defn:psieq}
			$\Psi_{eq,i} : QC^*(M)[[h]] \rightarrow CF^*(2\cdot H)[[h]]$ is defined by $$\Psi_{eq,i}\left( \sum_{j \ge 0} x_j h^j \right) = \sum_{j \ge 0} \sum_{y \in \mathcal{L}} n_{y,x_j,i} \cdot y \cdot h^j$$ where $n_{y,x,i}$ counts isolated $i$-equivariant spiked discs from $x$ to $y$ as above, with fixed parametrisation. Define $$\Psi_{eq}(x) = \sum_{i \ge 0} \Psi_{eq,i}(x) \cdot h^{i}.$$
		\end{defn}

We will use this for the proof of Theorem \ref{thm:qsseqpopintertwine} in Section \ref{sec:qsandeqpopclosed}. Using gluing and compactness arguments as in Section \ref{subsec:equivariantgluing} and \cite[End of Section 4]{seidel} yields the equations $$\sum_{j+k=i} d^k_{eq} \circ \Psi_{eq,j}(x) = \Psi_{eq,i}(dx)$$ for isolated solutions, which shows that $\Psi_{eq}$ descends to a map on equivariant cohomology.

Recall a classical lemma:

		\begin{lemma}
		\label{lemma:classicalpowerserieslemma}
			Let $\Lambda$ be a ring. An element $\sum_{i \ge 0} \lambda_i h^i$ of $\Lambda[[h]]$ is invertible when $\lambda_0$ is invertible.
		\end{lemma}
		\begin{proof}
			We write $$\sum_{i \ge 0} \lambda_i h^i = \lambda_0 \left( 1+ \sum_{i \ge 1} \lambda_0^{-1} \lambda_i h^i \right),$$ and a Taylor expansion of $(1+ b)^{-1}$ exists, where $b = \sum_{i \ge 1} \lambda_0^{-1} \lambda_i h^i$.
		\end{proof}

		\begin{lemma}
			$\Psi_{eq}$ is an isomorphism.
		\end{lemma}
		\begin{proof}
		
		Note $\Psi_{eq,0} = \Psi$ which is an isomorphism. Use Lemma \ref{lemma:classicalpowerserieslemma}. 
		\end{proof}

One can also construct the inverse map $\Psi_{eq}^{-1}$ directly, using reversed spiked discs as in Figure \ref{fig:psieqiso}(I). Specifically, we require almost complex structures $J^{R \Psi}_{v,z} \in \mathcal{J}_2$ on $M$ with the conditions that: 
		\begin{enumerate}
			\item $J^{R \Psi}_{v,s,t} = J$ for $s \le 0$.
			\item $J^{R \Psi}_{-v,s,t} = J^{R \Psi}_{v,s,t+1}$. 
			\item There is $S_0 >0$ such that, $J^{R \Psi}_{v,s,t} = J_{eq,v,t}$ for $s \ge S_0$.
		\end{enumerate}
We can likewise define $J^w_{R \Psi,s,t} = J^{R \Psi}_{w(s),s,t}$, and the definition of the reversed $i$-equivariant spiked discs is as in Definition \ref{defn:psieq}, i.e. for $x \in \mathcal{L}$ and $y \in \text{crit}(f)$, a reversed $i$-equivariant spiked disc is a triple $(w,u, \alpha)$ where $w: \mathbb{R} \rightarrow S^{\infty}$ is a $-\nabla g$ flowline from $v^{i,\pm}$ to $v^{0,+}$, and there are smooth maps $u: \mathbb{C} \rightarrow M$ and $\alpha: (-\infty, 0] \rightarrow M$. The map $\alpha$ satisfies $\dot{\alpha}(t) = - \nabla f(\alpha(t))$ and $\alpha(-\infty) = y$. Now we parametrise $\mathbb{C} - \{ 0 \}$ as $(s,t) \in \mathbb{R} \times \mathbb{R}/2\mathbb{Z} \mapsto e^{\pi(2 s+i t)}$, and $u$ satisfies:
\begin{itemize}
\item $\partial_s u + J^w_{R \Psi,s,t} (\partial_t u - X_{H_s}) = 0 \text{ on } \mathbb{C}$,
\item $u(0) = \alpha(0)$,
\item ${\displaystyle \lim_{s \rightarrow \infty}} u(s,t) = \begin{cases} \begin{array}{ll}
 		x(t) \text{ if } \sigma = + \\
		x(t+1) \text{ if } \sigma = - \\
                \end{array} \end{cases}$
\end{itemize}

To prove that the map $A$ counting reversed $i$-equivariant spiked discs as in Figure \ref{fig:psieqiso}(I) gives an inverse for $\Psi_{eq}$, use equivariant gluing to give setups as in Figure \ref{fig:psieqiso}(II)-(III) and argue as in the nonequivariant case.

\begin{lemma}
\label{lemma:inverseeqpsi}
$A = \Psi_{eq}^{-1}$.
\end{lemma}
\begin{proof}
We argue as in \cite[Theorem 4.1]{PSS}. After equivariant gluing, $A \circ \Psi_{eq}$ is counted by setups as in Figure \ref{fig:psieqiso}(II). A dimension argument, as in the nonequivariant case for weakly monotone symplectic manifolds, shows that there are not any nontrivial spheres. Specifically, for the coefficient of $y$ in the image of $x$ under the operation in (II), we count perturbed $J$-holomorphic spheres (i.e. the almost complex structure $J = J_{v,z}$ depends on $v \in S^{i}$ and $z \in S^2$, and the equation that $u$ satisfies is perturbed by a Hamiltonian). In particular, we count the number of perturbed holomorphic spheres intersecting the unstable manifold of $x$ and the stable manifold of $y$. However, after gluing we may homotope the $J_{v,z}$ (for $z \in S^2$) so that $J$ is independent of $v$. Neglecting transversality concerns momentarily, now we decouple the Floer theory on $M$ and the Morse theory on $S^{\infty}$. By the dimension argument as in \cite[8.1.4(ii)]{jhols}, we see that generically there are no such spheres (and as the set is empty, this data is trivially transverse even in the equivariant case). 

After gluing, $\Psi_{eq} \circ A$ is counted by setups as in Figure \ref{fig:psieqiso}(III). This is the equivariant continuation map from $H_0$ to $H_0$, hence the identity. The figures in (III) correspond to:
\begin{itemize}
\item taking a homotopy from the finite length Morse trajectory to a point,
\item gluing the two discs to give a cylinder,
\item homotoping the data so that the glued almost complex structure becomes $J_{v,s,t} = J_{eq,t}$ for all $s$.
\end{itemize}
For the homotopy in the third bullet point, we want to choose $J_{v,s,t,\eta}$ for $\eta \in [0,1]$ such that $J_{v,s,t,0} = J_{v,s,t}$ and $J_{v,s,t,1} = (J_{eq,v})_t$ such that $J_{-v,s,t+1,\eta} = J_{v,s,t,\eta}$ for all $\eta$. We prove that we can do this, in the same way as Lemma \ref{lemma:lemPQQS}. Specifically, let $\euscr{J}_1 \subset \euscr{J}$ be the space of such allowable homotopies inside the space of all homotopies from $J_{v,s,t}$ to $J_{eq,v}$. If $\rho_*$ is the action on $\euscr{J}$ induced from $\mathcal{J}_2$ then as $\euscr{J}_1$ and $\rho_* \euscr{J}_1$ are open and dense in $\euscr{J}$, so is their intersection. Choose a homotopy from this intersection.
\end{proof}

		\begin{figure}
			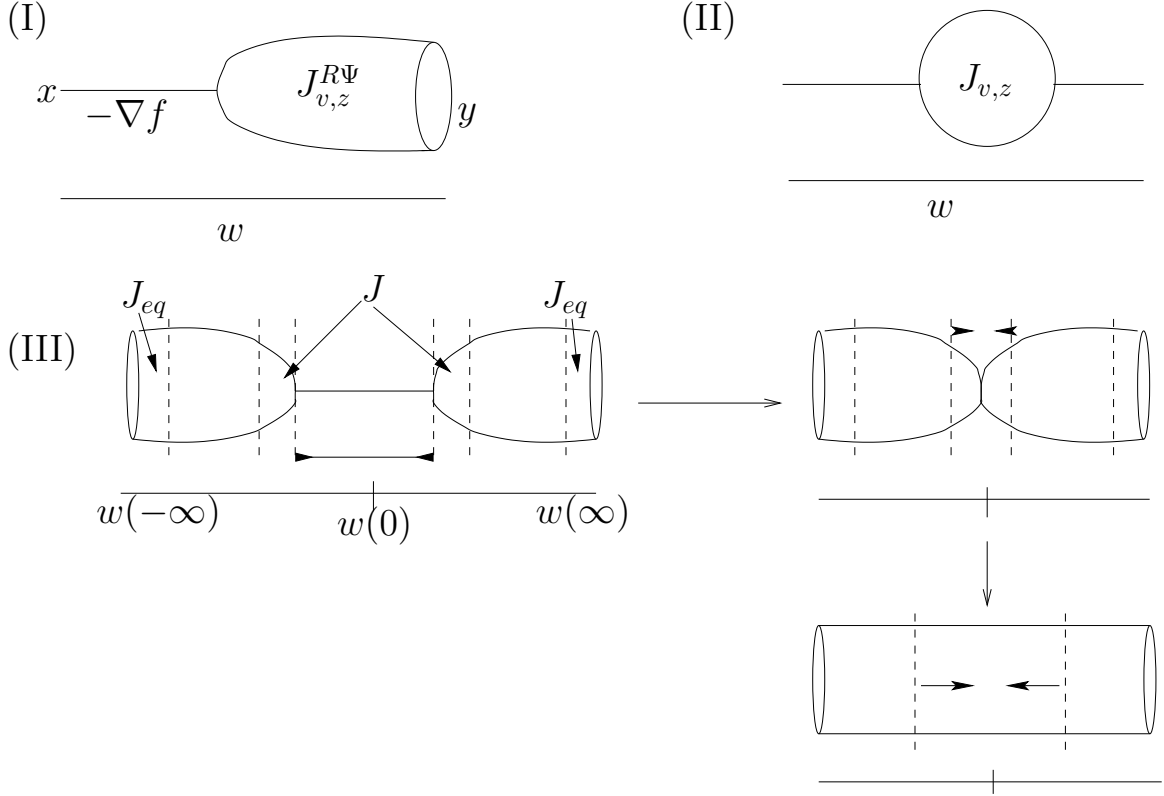
			\caption{Reversed $i$-equivariant spiked discs and gluing.}
			\label{fig:psieqiso}
		\end{figure}

\begin{rmk}
The argument in \cite[8.1.4(ii)]{jhols}, as referenced in Lemma \ref{lemma:inverseeqpsi} for the gluing in (II), is a rigorous proof of the intuitive argument that ``a pseudoholomorphic sphere with two marked points is unstable, even for a domain dependent almost complex structure".
\end{rmk}

		\begin{rmk}
		To choose a $J^{\Psi}_{v,z}$ as above, we must choose a smooth nullhomotopy $J^{\lambda}_{v,t}$ for $\lambda \in [-S_0 - 1, 1]$ from $J_{eq,v,t}$ to $J$, satisfying $J^{\lambda}_{v,z} = J_{eq,v,t}$ for $\lambda \in [-S_0 -1, -S_0]$ and $J^{\lambda}_{v,z} = J$ for $\lambda \in [0, 1]$. Then we can define $J^{\Psi}_{v,s,t} = J^{s}_{v,t}$ for $s \in  [-S_0 - 1, 1]$, and extend appropriately. One can proceed iteratively, because the map $v \in S^i \mapsto (J_{eq,v,t})$ is an element of $\pi_i$ of the space of compatible $J$, which is contractible. Thus we get a filling disc for each such $v \in S^i \mapsto (J_{eq,v,t})$, and the radial coordinate of this filling disc provides the $\lambda$ coordinate of the homotopy $J^{\lambda}_{v,t}$ (with radius $0$ corresponding to $\lambda = 0$). Then one perturbs this so that $J^{\Psi}_{v,z}$ depends smoothly on its factors. 
		\end{rmk}

\section{$Q\mathcal{S}$ and the Equivariant pair-of-pants for a $C^2$-small Hamiltonian}
\label{sec:qsandeqpopclosed}

		\begin{proof}[Proof of Theorem \ref{thm:qsseqpopintertwine}]
		
		For $x_{\pm} \in \text{crit}(f)$, then the coefficient of $y \in \text{crit}(f)$ in $\Psi_{eq}^{-1} \circ P (\Psi(x_{+}) \otimes \Psi(x_{-}))$ is calculated using setups as in the top of Figure \ref{fig:2Seidel}. Specifically, for each $i \ge 0$ we require sextuples $(w_L, w_M,u_L,u_M, u_{R,+},u_{L,+})$ such that:
\begin{itemize}
\item $w_L$ is a $- \nabla g$ flowline from $v^{i,\sigma_1}$ to $v^{j,\sigma_2}$, for some $0 \le j \le i$ and $\sigma_1, \sigma_2 \in \{ \pm 1 \}$.
\item $w_M$ is a $- \nabla g$ flowline from $v^{j,\sigma_2}$ to $v^{0,+}$,
\item $u_L$ is a reversed $(i-j)$-equivariant spiked disc, as in Section \ref{sec:eqpssisom},
\item $u_M$ satisfies Equation \eqref{equation:eqPOP} with respect to $w_M$,
\item $u_{R,+}$ and $u_{R,-}$ are spiked discs, both using the same $J, H_s$ and choice of $f^2_s$ perturbing the Morse function $f$, for $s \in [0,\infty)$ (where $f^2_s = f$ for $s \gg 0$).
\item $$\lim_{s \rightarrow -\infty} u_{R,\pm}(s,t) = \lim_{s \rightarrow \infty} u_{M}(\delta_{\pm}(s,t)).$$
\item $$\lim_{s \rightarrow \infty} u_{L}(s,t) = \lim_{s \rightarrow -\infty} u_{M}(\epsilon_{+}(s,t)).$$
\end{itemize}

We use the equivariant gluing argument, specifically as at the end of Section \ref{subsec:equivariantgluing}, on the right hand side: observe that near the right hand cylindrical ends of the pair of pants used for $u_M$, we have that $J^+_{right,v,s,t} = J^-_{right,v,s,t} = J_t$ for some $J_t$, and on the left hand cylindrical ends of the $0$-equivariant spiked discs $u_{R,\pm}$, we use the data $J^w_{s,t} = J_t$ for the same $J_t$. This is true for a sufficiently large gluing parameter. After gluing, we see that the coefficient of $y$ in $\Psi_{eq}^{-1} \circ P(\Psi(x_{+}) \otimes \Psi(x_{-}))$ is calculated using domains as in the middle figure of \ref{fig:2Seidel}. We may also use the equivariant gluing theorem for the left hand broken equivariant Floer trajectory. Hence one in fact counts configurations as in the bottom of Figure \ref{fig:2Seidel}.

		\begin{figure}
			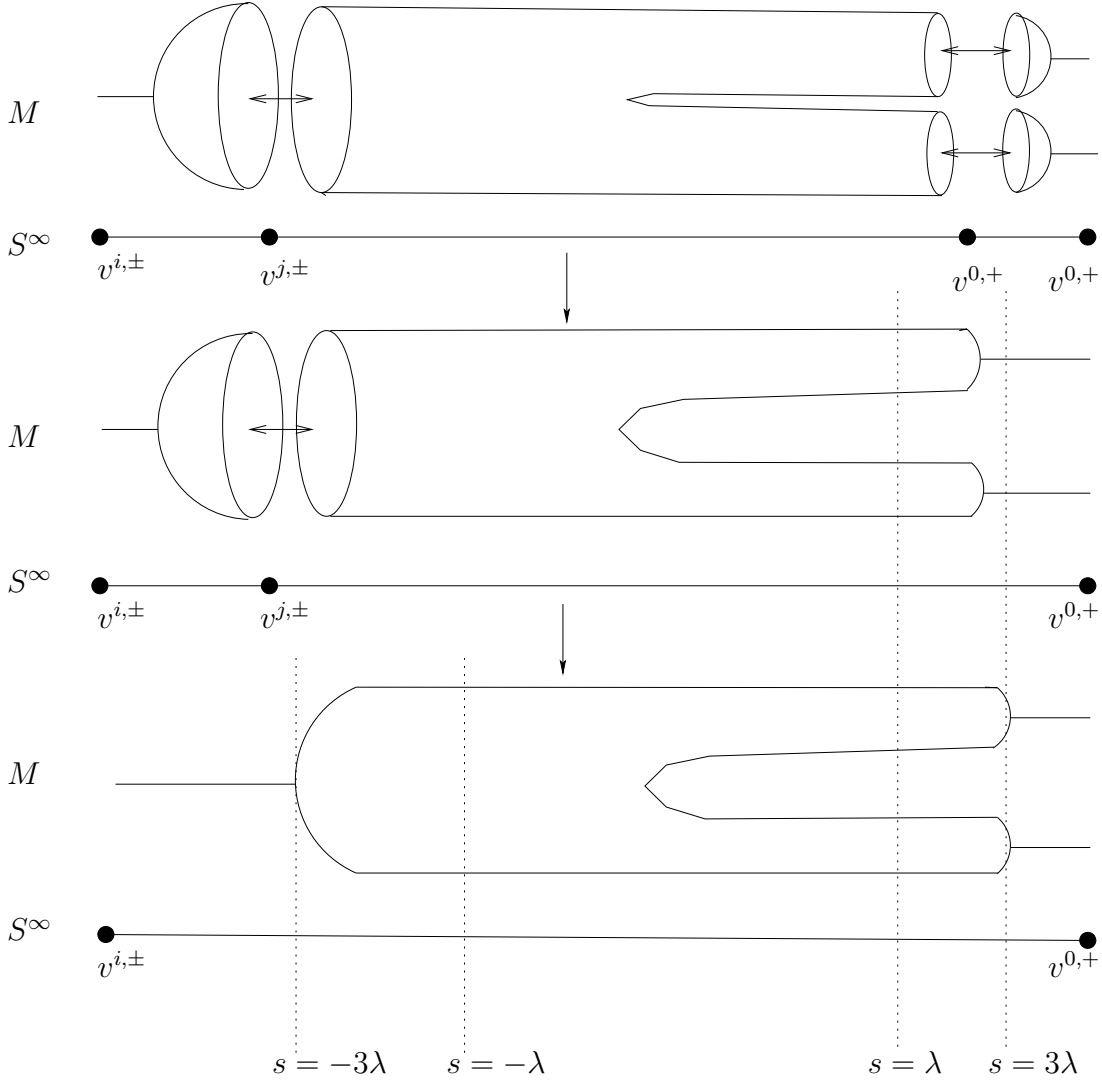
			\caption{Gluing of spiked disks onto equivariant pair of pants. Double headed arrows denote that the Hamiltonian loops are the same.}
			\label{fig:2Seidel}
		\end{figure}

Concretely, let $\tilde{S}$ consist of a pair-of-pants (with finite ends) truncated to the region $s \in [-\lambda,\lambda]$, with a $\tfrac{1}{\pi}\log(2 \lambda)$ radius reversed spiked disc $D_L$ attached on the left and two $\tfrac{1}{\pi}\log(2 \lambda)$ radius spiked discs $D^{\pm}_R$ attached on the right hand ends. For $\lambda>0$ sufficiently large, and for $w: \mathbb{R} \rightarrow S^{\infty}$ a $-\nabla g$ flowline, we may define $J^w_{z}$ for $z \in \tilde{S}$ as: 
\begin{itemize}
\item $J^w_z = J^w_{R \Psi,s',t'}$ for $z \in D_L$ parametrised using logarithmic coordinates on the left capping disc $(s',t') \mapsto \exp(2 \pi(s'+i t'))$,
\item $J^w_{z}$ satisfies the conditions in Section \ref{subsec:eqpop} for $z$ in the region $s \in [-\lambda,\lambda]$,
\item $J^w_z = J^w_{\Psi,s',t'}$ for $z \in D^{\pm}_R$ using logarithmic coordinates on the right hand capping discs $(s',t') \mapsto \exp(-2 \pi(s' + i t'))$.
\end{itemize}
If the gluing parameter $\lambda$ is sufficiently large then this is well defined, i.e. $J^w_z = J_t$ in some neighbourhood of $s = \pm \lambda$. We conclude that the coefficient of $y \in H^*(M)$ in $\Psi_{eq,i-k}^{-1} \circ P \mathcal{S}_k \circ \Psi(x)$ for $x \in H^*(M)$ is the number of isolated pairs $(w,u)$ such that:
\begin{itemize}
\item $w$ is a $-\nabla g$ flowline on $S^{\infty}$ from $v^{i,\pm}$ to $v^{0,+}$,
\item $u: \tilde{S} \rightarrow M$,
\item $(du - X_H \otimes \beta + Y) \circ j = J^w_z (du - X_H \otimes \beta + Y)$,
\item defining $\alpha_L$ as the restriction of $u$  to the negative halfline in $\tilde{S}$, then $\alpha_L: (-\infty, -3 \lambda] \rightarrow M$ satisfies $\dot{\alpha_L}(s) = - \nabla f (\alpha_L(s))$ and $\lim_{s \rightarrow -\infty} \alpha_L(s) = y$.
\item defining $\alpha_{R,\pm}$ as the restrictions of $u$  to the two respective positive halflines in $\tilde{S}$, then $\alpha_{R,\pm}: [3 \lambda, \infty) \rightarrow M$ satisfies $\dot{\alpha_{R,\pm}}(s) = - \nabla f^2_s (\alpha_{R,\pm}(s))$ and $\lim_{s \rightarrow \infty} \alpha_{R,\pm}(s) = x_{\pm}$.
\end{itemize}
This is a transverse setup for us to use the same Morse function $f^2_{s}$ for both of the right-hand Morse flows as we have picked a generic choice of $J_{v,z}$ and constant spheres may not appear due to the Hamiltonian perturbation. 

		\begin{lemma}
		\label{lemma:lemPQQS}
			The number of these glued solutions (as in the bottom of Figure \ref{fig:2Seidel}) for the coefficient of $y$ in $\Psi_{eq}^{-1} \circ \mathcal{P}(\Psi(x) \otimes \Psi(x))$ is the coefficient of $y$ in $Q\mathcal{S}(x)$.
		\end{lemma}
		\begin{proof}
			Choose $f^2_{v,s}$ as for the Morse quantum Steenrod square, see Section \ref{subsec:SqQviaMorse2}. We seek a homotopy $f^2_{v,s,\eta}$ for $\eta \in [0,1]$ with $f^2_{v,s,0}=f^2$ and $f^2_{v,s,1} = f^2_{v,s}$. Let $f^3_s := f^2_s$. Given a homotopy $f^2_{v,s,\eta}$, let $f^3_{v,s,\eta} := f^2_{-v,s,\eta}$. It is not immediate that generically if $f^2_{v,s,\eta}$ is an allowable homotopy from $f^2$ to $f^2_{v,s}$ then $f^2_{-v,s,\eta}$ is an allowable homotopy from $f^2$ to $f^2_{-v,s}$. 

			Let $F^{\lambda,i}_{ v,s} = (1-\lambda)f^i_{s} + \lambda f^i_{v,s}$ for $i=2,3$. This is some homotopy between the original choice of $f^{i}$ and $f^i_{s,v}$, which is $\mathbb{Z}/2$-equivariant. In general we may need to perturb this homotopy for transversality. Let $\mathcal{B}$ be the space of allowable perturbations of $F^{\lambda} = F^{\lambda,2}_{v,s}$. Let $\mathcal{B}_{i}$ be allowable perturbations of the Morse functions for the upper/lower leg, for $i=2,3$ respectively, of the setup as in the bottom of Figure \ref{fig:2Seidel}. These $\mathcal{B}_{i}$ are Baire sets in $\mathcal{B}$. There is an involution $\iota$ of $\mathcal{B}$ induced by $v \mapsto -v$. As $\mathcal{B}_i$ is Baire for $i=2,3$, so too is $\mathcal{B}_2 \cap \iota (\mathcal{B}_3)$. Choose any $F'^{\lambda}_{v,s}$ as the homotopy for $f^2$ and $F'^{\lambda}_{-v,s} = \iota F'^{\lambda}_{v,s}$ as the homotopy for $f^3$: this is an appropriate perturbation as $\mathcal{B}_2 \cap \iota (\mathcal{B}_3)$ is Baire. 

We use the homotopies $F^{\lambda,2}_{v,s}$ for $\alpha_{R,+}$ and $F^{\lambda,2}_{-v,s}$ for $\alpha_{R,-}$ to assume that $f^2_s$ to be $v$-dependent. That is, we require after homotopy that the respective flowlines on the right hand side satisfy $\dot{\alpha}_{R,\pm}(t) = - \nabla f_{s,\pm w(s)}(\alpha_{R,\pm}(t))$ for $s \in [3 \lambda, \infty)$ (and similarly on the left hand side). We then homotope $Y_z$ and $\beta$ to $0$. After translating in $s$ we may assume $\alpha_{R,\pm}: [0,\infty) \rightarrow M$ and $\alpha_L: (\infty,0] \rightarrow M$, along with the relevant $s$-translation in $f^2_{v,s}$. This is Definition \ref{defn:mqssv2} (using $\mathcal{M}_{\text{Morse},i,A}$) of the quantum Steenrod square, as the $J_{v,z}$ that we obtain after gluing satisfies $$J_{v,z} = J_{-v, z/(z-1)},$$ by construction (see Remark \ref{rmk:gluingaction}).
		\end{proof}
Recall that $P \mathcal{S}(\Psi(x)) = \mathcal{P}(\Psi(x) \otimes \Psi(x))$. This concludes the proof of Theorem \ref{thm:qsseqpopintertwine}.	
\end{proof}

		\begin{rmk}
		\label{rmk:gluingaction}
			\begin{enumerate}
	\item There is a $\mathbb{Z}/2$ action on the broken solutions (top of Figure \ref{fig:2Seidel}), induced from the action of $\gamma$ on the pair-of-pants and the left capping disc, and swapping the two right hand capping discs (along with a half-rotation for both of them). This action induces an action on the glued solutions. Moreover it induces the \textit{same} $\mathbb{Z}/2$ action as in the definition of $Q\mathcal{S}$, by triple transitivity of $PSL(2, \mathbb{C})$. The induced $\mathbb{Z}/2$ action on glued solutions must be the same action as in the $Q\mathcal{S}$ case, as it is a biholomorphism and its action on three points is known. The $\mathbb{Z}/2$ action in the symplectic square case is not explicitly stated in \cite{seidel}, but can be deduced from Addendum 4.12 (or by inspection). It is that solutions $(w,u)$ biject with $(-w,u \circ \gamma)$, where $\gamma : S \rightarrow S$ is the covering involution of $S \rightarrow \mathbb{R}/\mathbb{Z}$.
	\item Our sphere $\tilde{S}$ in the proof of Theorem \ref{thm:qsseqpopintertwine} can be thought of as a conformal rescaling of the sphere $S^2$, in an annulus near each of the three points. Specifically, pick three points $z_1,z_2,z_3 \in S^2$, and $0<\epsilon<\epsilon'$ such that the $\epsilon'$ balls around the $z_i$ are pairwise distinct. We can think of the finite length cylindrical edges in $\tilde{S}$ as taking logarithmic coordinates on the annulus $B_{\epsilon'}(z_i) - \mathring{B}_{\epsilon}(z_i)$.
			\end{enumerate}
		\end{rmk}

\section{$Q\mathcal{S}$ and the symplectic square for equivariant symplectic cohomology}
\label{sec:qsandeqpopopen}

\subsection{Equivariant Symplectic Cohomology}
\label{subsec:equisymplcoh}
Symplectic cohomology is formed as the direct limit of Floer cohomologies using continuation maps $\phi_{H, H'}: HF^*(H) \rightarrow HF^*(H')$ where $H,H'$ are radial at infinity with $\lambda_H \le \lambda_{H'}$. Specifically, $$SH^{*}(M) := \underset{H}{\varinjlim} \ HF^{*}(H),$$ where the direct limit is taken over all Hamiltonians $H$ that are radial at infinity. 

We want to modify this to define $SH^*_{eq} (M)$. There are different ways of doing this in general, but we work as suggested in \cite[Equation (2.50)]{seidel}. Specifically we use the equivariant continuation maps we defined in Section \ref{subsec:equivariantcontinuationmaps} to define the direct limit of the directed set $\{ HF_{eq}^*(2 \cdot H) \}$, where the $H$ are radial at infinity, so $$SH^*_{eq}(M) := \underset{H}{\varinjlim} \ HF_{eq}^*(2 \cdot H).$$ We will define a symplectic square \begin{equation} \label{equation:PSSH} P \mathcal{S}: SH^{*} (M) \rightarrow SH_{eq}^{2*} (M) \end{equation} induced by the symplectic square on Floer cohomology. In order to be well defined, this requires Lemma \ref{lemma:pairofpantsandcontin}:

\begin{proof}[Proof of Lemma \ref{lemma:pairofpantsandcontin}]
		\begin{figure}
			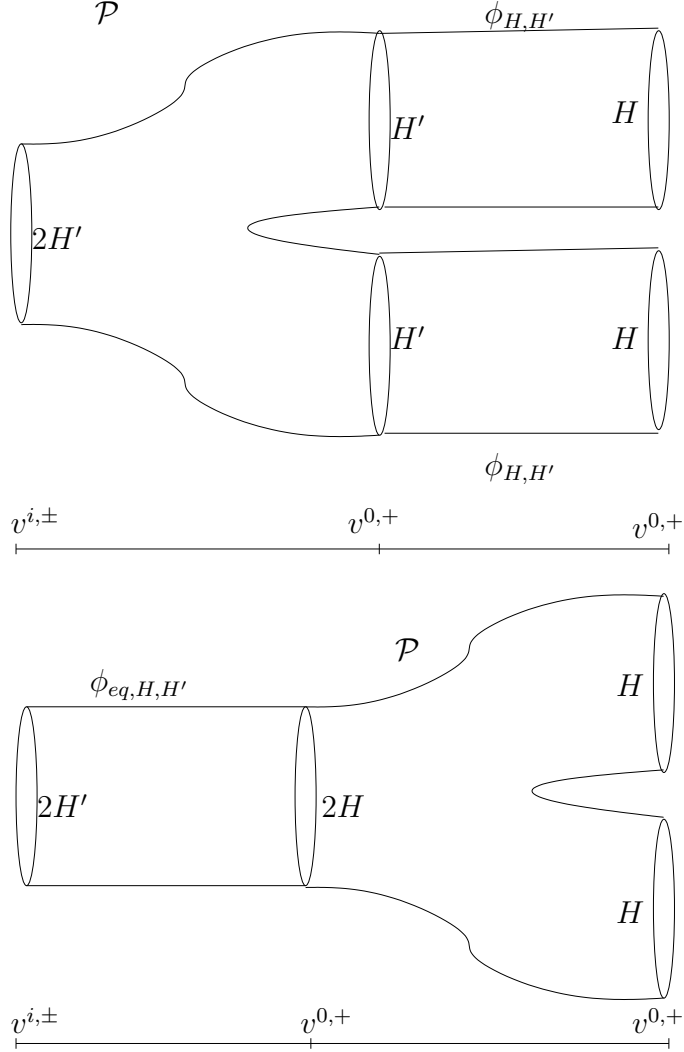
			\caption{Broken solutions for $P \mathcal{S} \circ \phi_{H,H'}$ and $\phi_{eq,H, H'} \circ P \mathcal{S}$ respectively.}
			\label{fig:eqpopandcontin}
		\end{figure}

	Using Figure \ref{fig:eqpopandcontin}, apply the equivariant gluing theorem on the left hand cylinder of the lower setup, and the two right hand legs of the upper setup. The proof follows by invariance under $\mathbb{Z}/2$-equivariant homotopies of the choice of $J$ (which for an appropriate choice of $J$ and gluing parameter may be a constant homotopy), and the other auxiliary data. The flowline from $v^{i,\pm}$ to $v^{j,\sigma}$ (for $j \le i$) in the lower part of Figure \ref{fig:eqpopandcontin} could be rephrased as a solution for the asymptotics $v^{i-j,\pm \sigma}$ and $v^{0,+}$, as in Remark \ref{rmk:shiftandnegativeaction}. 
\end{proof}

Lemma \ref{lemma:pairofpantsandcontin} allows us to define $P \mathcal{S}$ in \eqref{equation:PSSH} as the direct limit of the Hamiltonian Floer case. Let $\xi_{eq}: HF^*_{eq}(2 \cdot H) \rightarrow SH_{eq}^*(M)$ be the natural map to the direct limit.

\begin{defn}[The equivariant $c^*$-map]

$c^*_{eq} := \xi_{eq} \circ \Psi_{eq}$.

\end{defn}

\begin{proof}[Proof of Corollary \ref{corollary:qsseqpopintertwine2}]
Use Theorem \ref{thm:qsseqpopintertwine} and descend to the direct limit.
\end{proof}

\section{The symplectic Cartan relation}
\label{sec:symplcartan}
There is no immediate analogue of the (quantum) Cartan relation for the symplectic square. This is because there is no obvious product $*_{eq}$ on $SH^*_{eq}(M).$ A symplectic Cartan relation might have been expected to take the form $$P \mathcal{S}(x*y) = P \mathcal{S}(x) *_{eq} P \mathcal{S}(y) + q(x,y),$$ for an appropriate product $*_{eq}$ on $SH^*_{eq}(M)$ and a correction term $q: SH^*(M) \otimes SH^*(M) \rightarrow SH^*_{eq}(M)$ as in \cite[Theorem 1.2]{me}. A product is not an issue in the quantum and classical cases because for example $QH^*_{eq}(M) = QH^*(M)[[h]]$, which has an obvious product. This is not true for $SH^*_{eq}(M)$.

Recall in Remark \ref{rmk:shiftandnegativeaction} that it was important for defining the symplectic square that we use an almost complex structure $J^w_{z}$ such that $J^w_z = J^{-w}_{\gamma z}$, where $\gamma$ is an involutary biholomorphism of the pair-of-pants $S$. Specifically, we require that if $u: S \rightarrow M$ is $J^w$-holomorphic then $u \circ \gamma : S \rightarrow M$ is $J^{-w}$-holomorphic, and the two solutions are related by a $\mathbb{Z}/2$ action. Attempting to do the same for $*_{eq}$ fails because we would need to choose $J^w_z$ for $z \in S$ such that near infinity, on the cylindrical ends, $J^w_z = (J_{eq,w(s)})_t$ for $z=(s,t)$ in cylindrical coordinates. This is to ensure that we are constructing a chain map (consider for example what one must obtain in a compactification of the moduli space). Suppose for a contradiction that $J^{-w}_z = J^w_{\gamma' z}$ for some involutary biholomorphic $\gamma' : S \rightarrow S$. Then $\gamma'$ would have to be a half-rotation near each of the three cylindrical ends. As we may view the pair-of-pants as a three punctured sphere (using logarithmic coordinates), we may consider $\gamma'$ to be a biholomorphism of the $3$-punctured sphere. The condition that $\gamma'$ must be a half-rotation near the cylindrical ends means that $\gamma'$ is bounded near each of the punctures, so it extends to a biholomorphism $\gamma' : S^2 \rightarrow S^2$, which fixes each of the three points that were removable singularities. A biholomorphism of $S^2$ that fixes three points is the identity, giving a contradiction, so no such $\gamma'$ exists.

\subsection{The symplectic Cartan relation}
\label{subsec:symplcart}

In this section we will define some almost complex structures $J^p_{v,s,t}$ where $(s,t)$ parametrises a half-cylinder as given below, $v \in S^{\infty}$ and $p$ is an edge label. We give some important terminology to save on repeating ourselves:

\begin{enumerate}
\item $J^p_{v,s,t}$ is incoming if $s \in (-\infty,0]$ and outgoing if $s \in [0,\infty)$.
\item $J^p_{v,s,t}$ is equivariant if $t \in \mathbb{R}/2 \mathbb{Z}$ and nonequivariant if $t \in \mathbb{R}/\mathbb{Z}$.
\item $(J^p_{v,s,t},J^q_{v,s,t})$ are symmetric if $J^p_{v,s,t}=J^q_{-v,s,t}.$
\end{enumerate}

Further:
\begin{enumerate}
\item If $J^p_{v,s,t}$ is equivariant then it satisfies $J^p_{v,s,t} = J^p_{-v,s,t+1}$ for all $(v,s,t)$.
\item If $J^p_{v,s,t}$ is equivariant then $J^p_{v,s,t} = J_{eq,v,t}$ for $|s| \gg 1$.
\item If $J^p_{v,s,t}$ is nonequivariant then $J^p_{v,s,t} = J_t$ for $|s| \gg 1$, where $J_t = J^{\pm}_{right,v,2,t}$ as in Section \ref{subsec:eqpop}.
\item For all $p$, the $J^p_{v,s,t} = J$ for $0 \le |s| \le 1$ for some fixed almost complex structure $J$ on $M$.
\end{enumerate}

We choose almost complex structures:

\begin{itemize}
\item $(J^{LL}_{v,s,t}, J^{LR}_{v,s,t})$ to be outgoing, nonequivariant and symmetric.
\item $(J^{UL}_{v,s,t}, J^{UR}_{v,s,t})$ to be outgoing, nonequivariant and symmetric.
\item $J^{D}_{v,s,t}$ to be incoming and equivariant.
\end{itemize}

		\begin{figure}
			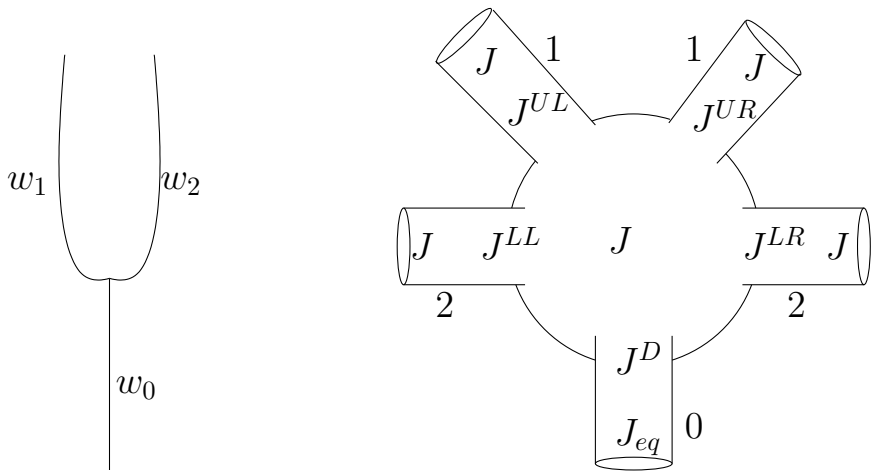
			\caption{The almost complex structure $J^w_z$. The labels $0$, $1$, $2$ next to the edges indicate which $w_i$ is being used in that situation.}
			\label{fig:acssympcart}
		\end{figure}

As we work with cohomology, ``outgoing" ends correspond to inputs and ``incoming" ends correspond to outputs. For each $R \in (1,\infty)$ and $p \in \{ LL,LR,UL,UR,D \}$, let $Z_p \in S^2$ such that: $$Z_{LL} = 0, \ Z_{LR} = \infty, \ Z_{D} = 1, \ Z_{UL} = -1/R,  \ Z_{UR} = -R.$$ 

For $p \in \{ LL,LR,UL,UR,D \}$ pick $\theta_p(R) \in ST_{Z_p} S^2$, the unit circle bundle of $S^2$ at $Z_p$, such that $\theta_{LL} = i^* \theta_{LR}$ and $\theta_{UL} = i^* \theta_{UR}$, where $i(z) = 1/z$. Pick smooth functions $r_p: (1,\infty) \rightarrow \mathbb{R}_{> 0}$ such that the discs $B_{r_p(R)}(Z_p)$ are pairwise distinct for each $R$, and $r_p(R) = r_{p'}(R)$ for $(p,p') = (UL,UR),(LL,LR)$. We let $m_R$ be the associated Riemann surface using logarithmic coordinates on $B_{r_p(R)} (Z_p) - \{ Z_p \}$. The parametrisation of the cylindrical ends for each $p$ are determined by the relevant descriptors ``outgoing, incoming" etc for the corresponding $J^p$ above. 

Let $w = (w_0,w_1,w_2)$ be flowlines on $S^{\infty}$ with respect to perturbations $g_{0,s}, g_{1,s}, g_{2,s}$ of the Morse function $g$, defined for $s \in (-\infty,0]$ for $g_{0,s}$ and $s \in [0,\infty)$ for $g_{1,s}, g_{2,s}$. This is illustrated in Figure \ref{fig:equivariantproductgr}. These $g$ must be chosen to satisfy:

\begin{itemize}
\item Each $g_{q,s} = g$ for $|s| \gg 0$,
\item the ${g}_{q,0}$ are generic, to ensure transverse moduli spaces,
\item $g_{q,s}(\tau x) = g_{q,s}(x) + \text{constant}$, so that $\tau^i$ induces a bijection of flowlines with endpoints $v^{i,+}$ and $v^{0,+}$ as in Remark \ref{rmk:shiftandnegativeaction},
\item $g_{q,s}(-x) = g_{q,s}(x)$, so that they descend to functions on $\mathbb{RP}^{\infty}$.
\end{itemize}

		\begin{figure}
			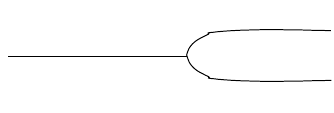
			\caption{Flowlines on $S^{\infty}$.}
			\label{fig:equivariantproductgr}
		\end{figure}

One should think of these $g_{q,s}$ as being used analogously to calculating the cup product on $\mathbb{RP}^{\infty}$ using Morse theory (and is equal to this when restricting to each $\mathbb{RP}^i$). An example of such $g_{i,s}$ is:
\begin{itemize}
\item $g_{0,s} = g_{1,s} = g$ for all $s$,
\item Let $\overline{g}(x_0,x_1,\ldots) = g(x_0,x_1,\ldots) + \epsilon \sum_{j \ge 0} x_{2j} x_{2j+1}$ for some small $\epsilon$. Then let $g_{2,s}(v) = g(v) + \epsilon (1 - \beta(s)) \overline{g}(v)$ where $\beta: [0,\infty) \rightarrow \mathbb{R}$ is smooth, nonincreasing, and $\beta(s) = 1$ for $s \ge 1$ and $\beta(s) = 0$ for $s \le 1/2$. This satisfies all of the conditions above: crucially, it ensures the moduli spaces that we count will be transverse.
\end{itemize}
We define $J^w_{z}$ for $z \in m_R$ such that \begin{equation} \label{equation:acsops} J^w_{s,t} = J_{p,w_{q}(s),s,t} \end{equation} for $z=(s,t)$ on the cylindrical end $e_p$ associated to $Z_p$, and $J_{z} = J$ away from the cylindrical ends, and for pairs $$(p,q) = (D,0), (UL,1), (UR,1), (LL,2), (LR,2).$$ We see that $J^w_z = J^{-w}_{1/z}$.

Given $x,y,z \in SC^*(M)$, $i \in \mathbb{Z}$ and $R \in (1,\infty)$, let $\mathcal{M}_{i,R}(z; x,y)$ consist of pairs $(w,u)$ such that $w$ is a triple of flowlines as in Figure \ref{fig:equivariantproductgr}, with $w_0(-\infty) = v^{i,\sigma}, \ w_1(\infty) = v^{0,+}, \ w_2(\infty) = v^{0,\pm}$ for $\sigma \in \{\pm \}$, and $u: m_R \rightarrow M$ such that ${\displaystyle \lim_{s \rightarrow \infty}} u|_{e_p} (s,t) = \alpha(t)$ for $$(p,\alpha) = (D,z), (UL,y), (UR,y), (LL, x), (LR,x), $$ for $\sigma=+$. For $\sigma = -$, replace $z(t)$ by $z(t+1)$. Then $u$ satisfies \begin{equation} \label{equation:operationfloereq} (du - X_H \otimes \beta + Y) \circ j = J^w_z \circ (du - X_H \otimes \beta + Y)\end{equation} as in the standard Floer equation where $X_H$ is the Hamiltonian vector field associated to $H$ and $\beta$ is a $1$-form on $m_R$ with $d \beta =0$, and $\beta = b_1 dt$ on $e_{LL}, e_{LR}$ where $x \in CF^*(b_1 H)$, similarly $\beta = b_2 dt$ on $e_{UL}, e_{UR}$ and $\beta = (b_1 + b_2) dt$ on $e_D$, where $b_1,b_2 \in \mathbb{Z}_{> 0}$ (where $\beta$ varies smoothly with $R$). We generally have to perturb using a Hamiltonian $Y$, invariant under $z \mapsto 1/z$ and supported only on a compact subset of the cylindrical ends (in order to prevent constant solutions).

\begin{defn}
For generic $R \in (1,\infty)$, $i \ge 0$ $$q(R)^{i}: SC^*(M) \otimes SC^*(M) \rightarrow SC^*(M), \qquad q(R)^i(x,y) = \sum_{z \in \mathcal{L}} \# \mathcal{M}_{i,R}(z; x,y) \cdot z$$ where $\#$ counts isolated solutions. Then $$ x \otimes y \mapsto \sum_{i \ge 0} h^i q(R)^i(x,y)$$ descends to a well defined map $q(R): SH^*(M) \otimes SH^*(M) \rightarrow SH^*_{eq}(M)$, with a similar proof as in Section \ref{subsubsec:eqpopischainmap}.
\end{defn}

		\begin{figure}
			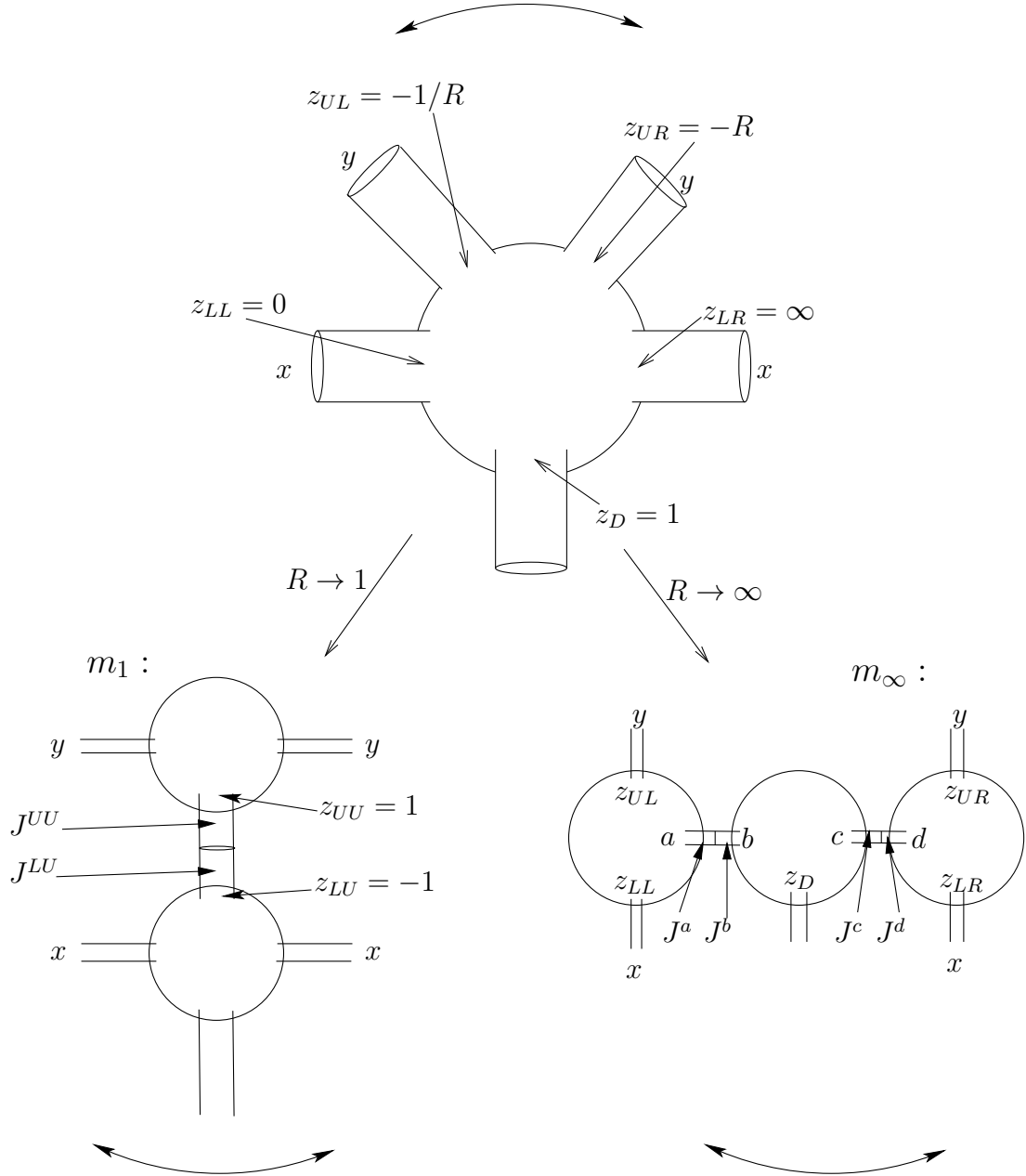
			\caption{The $1$-dimensional moduli space linking $m_1$ and $m_{\infty}$. The double-headed arrows denote $\mathbb{Z}/2$ symmetry. For simplicity we omit the $S^{\infty}$ data from this picture.}
			\label{fig:1modspaandbreak}
		\end{figure}

The $1$-cobordism associated to $R \in (1,\infty)$ is illustrated in Figure \ref{fig:1modspaandbreak}. Its boundary consists of operations corresponding to $m_1$ and $m_{\infty}$ (given in Figure \ref{fig:1modspaandbreak}).

For $m_1$, two new special points are introduced, $Z_{UU}, Z_{LU}$. These points correspond to cylindrical ends, and one must make a consistent choice of $\theta$ and radius for each end. We choose $J^{UU}_{v,s,t}$ to be incoming and equivariant and $J^{LU}_{v,s,t}$ to be outgoing and equivariant. In this case, for isolated solutions we use quartets of flowlines on $S^{\infty}$ denoted $w = (w_0,w_1,w_2,w_3)$ as in Figure \ref{fig:eqprod4us}(I), and our almost complex structure $J_{z}$ for $z \in m_1$ satisfies Equation \eqref{equation:acsops} for $$(p,q) = (D,0), (LU, 1), (LR,2), (LL,2), (UU, 3), (UL,3), (UR,3).$$

		\begin{figure}
			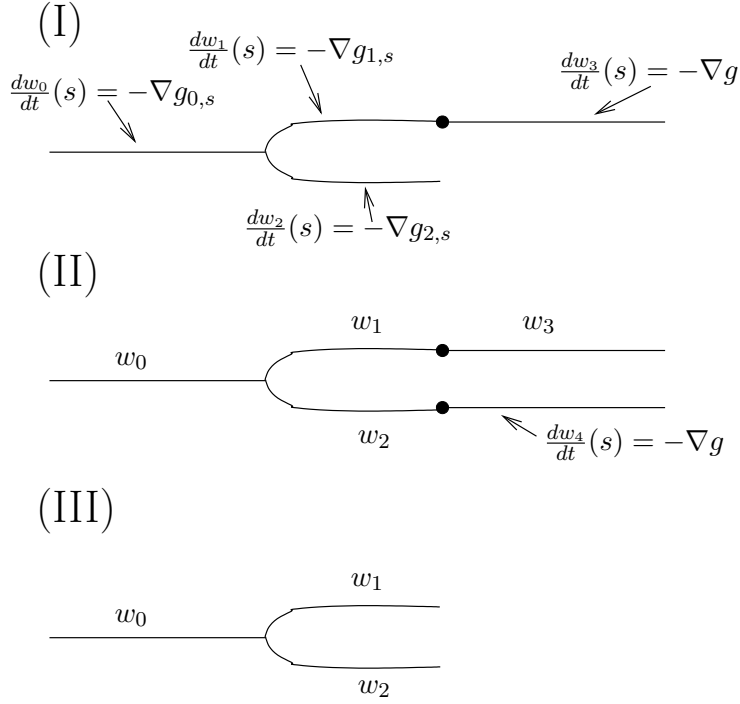
			\caption{Flowline setups used in $S^{\infty}$.}
			\label{fig:eqprod4us}
		\end{figure}

For $m_{\infty}$, there are four new edges that appear, which are labelled $a,b,c,d$ as in Figure \ref{fig:1modspaandbreak}. We choose almost complex structures such that:
\begin{itemize}
\item $(J^a_{v,s,t},J^d_{v,s,t})$ are incoming, nonequivariant and symmetric.
\item $(J^b_{v,s,t},J^c_{v,s,t})$ are outgoing, nonequivariant and symmetric.
\end{itemize}

In this case we use quintets of flowlines on $S^{\infty}$, denoted $w=(w_1,w_2,w_3,w_4,w_5)$ as in Figure \ref{fig:eqprod4us}(II). After homotoping the data as in Lemma \ref{lemma:symplcartan}, the only isolated solutions will be for $w_3$ and $w_4$ being constant flowlines at $v^{0,+}$ and $v^{0,\pm}$. So this really reduces to the setup of Figure \ref{fig:eqprod4us}(III).

The points in $\mathcal{M}_{i,1}(z;x,y) \sqcup \mathcal{M}_{i,\infty}(z; x,y)$ form the boundary of the $1$-dimensional cobordism $\bigsqcup_{R \in [1,\infty]} \mathcal{M}_{i,R}(z; x,y)$, hence $\# \mathcal{M}_{i,1}(z;x,y) = \# \mathcal{M}_{i,\infty}(z;x,y)$ and so $q(1)^i = q(\infty)^i$.

\begin{lemma}
\label{lemma:symplcartan}
	$\sum_i q(\infty)^i(x,y) h^i = P \mathcal{S}(x * y)$ where $*$ is the pair-of-pants product.
\end{lemma}
\begin{proof}
		Consider Figure \ref{fig:1modspaandbreak}, specifically the $m_{\infty}$ endpoint. Homotope the data (i.e. the almost complex structures) on the left and right spheres $\mathbb{Z}/2$-equivariantly, as we did for the Morse function in Lemma \ref{lemma:lemPQQS}, so that they are independent of $v$. It is then immediate that $q(\infty)^i(x,y) = P \mathcal{S}_i(x*y)$. 
\end{proof}

Let $\tilde{S}$ be the domain $m_1$, but removing the sphere with edges $Z_{UL}, Z_{UU}, Z_{UR}$, see Figure \ref{fig:operationpsdash}. Thus, $\tilde{S}$ is a Riemann surface with $1+3$ cyclindrical ends. Keep all almost complex structures and other conditions the same as for $m_1$.

		\begin{figure}
			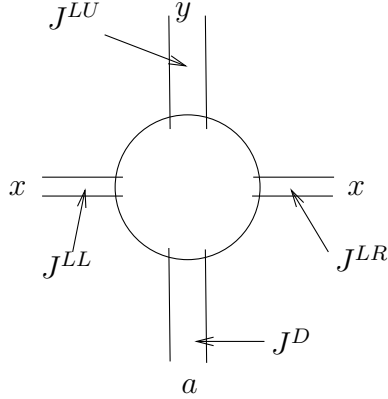
			\caption{Configurations for the coefficient of $a$ in the operation $P \mathcal{S}'(x;y)$.}
			\label{fig:operationpsdash}
		\end{figure}

\begin{defn}
\label{defn:kindasymplcart}
	Using the same edge labels and almost complex structures for $\tilde{S}$ as for $m$, define $$P \mathcal{S}'_i: CF^j(b_1 H) \otimes CF^k_{eq}(2 \cdot b_2 H) \rightarrow CF^{2j+k-i}_{eq}(2 \cdot (b_1+b_2)H),$$ $$P \mathcal{S}'_i(x;y h^l) = \sum_a n_{x,y,a,l,i} \cdot a$$ where $n_{x,y,a,l,i}$ counts the number of pairs $(w,u)$ where $w$ is a setup as in Figure \ref{fig:eqprod4us}(III) with $w_0(-\infty) = v^{i+l,\sigma_0}, \ w_1(\infty) = v^{l,\sigma_1}, \ w_2(\infty) = v^{0,+}$ and $u: \tilde{S} \rightarrow M$ satisfies Equation \eqref{equation:operationfloereq} on $\tilde{S}$ with:

\begin{equation}
\label{equation:asymptotics}
\begin{array}{lll}
            {\displaystyle \lim_{s \rightarrow -\infty}} u|_{e_{D}} &=&  \begin{cases} \begin{array}{ll}
 												a(t) \text{ if } \sigma_0 = + \\
												a(t+1) \text{ if } \sigma_0 = -
                										\end{array} \end{cases}
\\
           {\displaystyle  \lim_{s \rightarrow \infty}} u|_{e_{LU}} &=& \begin{cases} \begin{array}{ll}
 												y(t) \text{ if } \sigma_1 = + \\
												y(t+1) \text{ if } \sigma_0 = -
                										\end{array} \end{cases}
\\
           {\displaystyle  \lim_{s \rightarrow \infty}} u|_{e_{k}} &=& x(t) \text{ for } k=LR,LL.
\end{array}
\end{equation}

	Define $P \mathcal{S}'(x;y h^l) :=\sum_{i \ge 0} P \mathcal{S}'_i(x;y h^l) h^{l+i}.$
\end{defn}

Note that $P \mathcal{S}'$ descends to a map on cohomology as in Section \ref{subsubsec:eqpopischainmap}: first we construct  $$\mathcal{P}': (CF^*(b_1 H) \otimes CF^*(b_1 H)) \otimes CF^*_{eq}(2 \cdot b_2 H) \rightarrow CF^*_{eq}(2 \cdot (b_1+ b_2) H),$$ where for $\mathcal{P}'(x_{LL}, x_{LR}, y)$ we consider $u$ satisfying Equations \eqref{equation:operationfloereq} and \eqref{equation:asymptotics}, changing $x(t)$ to $x_k(t)$ in \eqref{equation:asymptotics}. We extend $h$-linearly to $$\mathcal{P}': C^*_{\mathbb{Z}/2}(CF^*(b_1 H) \otimes CF^*(b_1 H)) \otimes CF^*_{eq}(2 \cdot b_2 H) \rightarrow CF^*_{eq}(2 \cdot (b_1+ b_2) H).$$   We then analyse the different possible ways in which the cylindrical ends may break. This shows that $\mathcal{P}'$ is a chain map, hence $P \mathcal{S}' = \mathcal{P}' \circ (\eta \otimes id)$ is a well-defined map on cohomology.

It is immediate that $P \mathcal{S}'_0(x;y h^i)$ is a chain representative of $x^2 y h^i$ and $P \mathcal{S}'_i(x;1)$ is a chain representative of $P \mathcal{S}_i(x)$.

\begin{rmk}
	Using the $1$-dimensional cobordism illustrated in Figure \ref{fig:actiononSHeq} for $R \in [1,\infty]$, arguing as in Lemma \ref{lemma:symplcartan} it can be shown that $$P \mathcal{S}'(x*z;y) = P \mathcal{S}'(x;P \mathcal{S}'(z,y))$$ or rather that $P \mathcal{S}'$ defines a $SH^*(M)$-module structure on $SH^*_{eq}(M)$. This seems to be the best we can do with the given approach, considering there is no obvious product structure on $SH^*_{eq}(M)$.

		\begin{figure}
			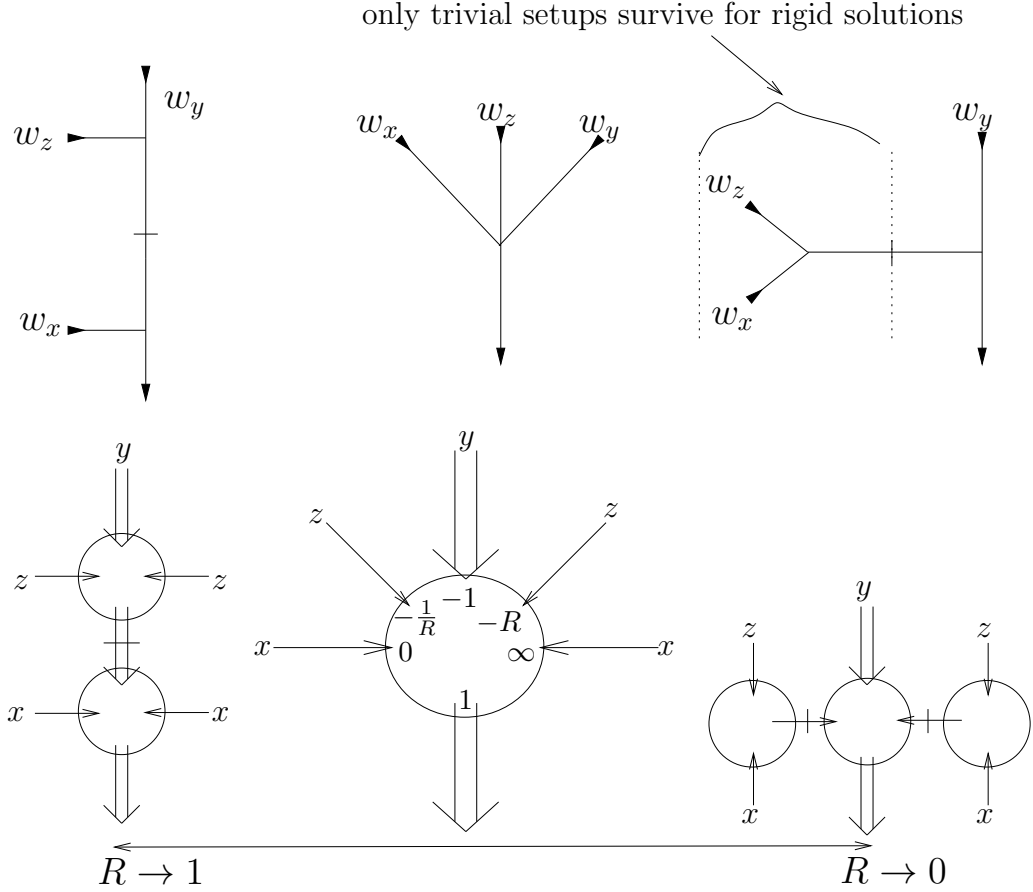
			\caption{For the lower row, single arrows represent Floer trajectories for inputs from $SH^*(M)$. Double arrows represent trajectories for inputs/outputs from $SH^*_{eq}(M)$. Arrows split by an orthogonal line are broken trajectories. The upper row demonstrates the relevant flowlines on $S^{\infty}$, labelled accordingly.}
			\label{fig:actiononSHeq}
		\end{figure}

\end{rmk}

\begin{lemma}
\label{lemma:psisq1}
$$\sum_i q(1)^i(x,y) h^i = P \mathcal{S}'(x;P \mathcal{S}(y)).$$
\end{lemma}
\begin{proof}
This is immediate, given the definition of $P \mathcal{S}'$.
\end{proof}

\begin{proof}[Proof of Theorem \ref{thm:symplecticcartan}]
	Use Lemmas \ref{lemma:psisq1} and Lemma \ref{lemma:symplcartan}. 
\end{proof}

\begin{rmk}[A sketch of the comparison to the quantum Cartan relation]
Consider the methods in \cite[Section 5]{me}, when proving the quantum Cartan relation. The operations $$(x,y) \mapsto \begin{cases} \begin{array}{l} Q \mathcal{S}(x) * Q \mathcal{S}(y)  \\ Q \mathcal{S}(x*y) \end{array} \end{cases}$$ correspond to two different configurations of five marked points on nodal spheres. To identify these two operations would have required a $\mathbb{Z}/2$-invariant homotopy between these two point configurations, and it can be shown that such a homotopy does not exist. Here, the $\mathbb{Z}/2$-action acts by as $(12)(34)$ if the marked points are labelled $z_0, z_1, z_2, z_3, z_4$.

In the referenced section of the paper, these configurations were embedded as points in a four-dimensional space of domains, and it was observed that while there was no $\mathbb{Z}/2$-invariant homotopy connecting these configurations, we could do the next best thing and find a $\mathbb{Z}/2$-equivariant cycle (i.e. an element of the $\mathbb{Z}/2$-equivariant homology of the space of domains) that allowed us to construct a correction term in the quantum Cartan relation. Attempting to do this in the symplectic case involves a much more complicated moduli space. Broadly, one has (at least) five punctures as opposed to marked points, along with the information of an asymptotic point at each puncture. The addition of the asymptotic point means that, under the compactification, some points in the moduli space the quantum case are replaced by copies of $S^1$ (because the induced asymptotic point on the limiting configuration may well depend on the direction at which it is approached). Depending on the choice of parametrisation, the point ``corresponding" to the operation $(x,y) \mapsto Q \mathcal{S}(x) * Q \mathcal{S}(y)$ is replaced by a copy of $S^1$. This prevents us from defining the product of the symplectic squares. Moreover, one might consider using this entire copy of $S^1$ to define an operation, but as it is the boundary of a $2$-disc in the moduli space the induced operation will be $0$.
\end{rmk}

\subsection{Cotangent bundles for $n$-spheres}
\label{subsec:calcs}

Seeliger in \cite{seeliger} calculates $\mathbb{H}_{*}(\mathcal{L}S^n; \mathbb{Z}): = H_{*+n}(\mathcal{L}S^n; \mathbb{Z})$ using the Serre fibration $\Omega S^n \xhookrightarrow{} \mathcal{L}S^n \rightarrow S^n$, where $\Omega S^n$ is the based loop space of $S^n$ and $\mathcal{L}S^n = C^{\infty}(S^1,S^n)$ is the free loop space of $S^n$. We can take a tensor product of the spectral sequence with $\mathbb{Z}/2$ to show that with $\mathbb{Z}/2$ coefficients: $$\mathbb{H}_* (\mathcal{L}S^n) \cong \mathbb{H}_*(S^n) \otimes H_*(\Omega S^n)$$ as rings, where $\mathbb{H}_*(S^n) = H_{*+n}(S^n)$. The ring structure uses the Chas-Sullivan loop product, the intersection product and the Pontrjagin product respectively. Then $$\mathbb{H}_*(\mathcal{L}S^n) \cong \tfrac{\mathbb{Z}}{2}[x] / (x^2) \otimes \tfrac{\mathbb{Z}}{2}[y],$$ where $|x| = -n$ and $|y|=n-1$.

The Viterbo isomorphism states that $SH^*(T^* S^n) \cong \mathbb{H}_{-*}(\mathcal{L}S^n)$ as rings, see \cite{viterbo}, so these $x,y \in \mathbb{H}_*(\mathcal{L}S^n)$ correspond (abusively) to $|x| = n$ and $|y| = 1-n$ in $SH^*(T^* S^n)$. For example, for odd $n$ we may use as representatives of the generators of $H_*(\mathcal{L}S^n)$ embedded submanifolds as given by Oancea \cite[Section 7]{oancea}, for $n>1$.

By Theorem \ref{thm:symplecticcartan}, to calculate $P \mathcal{S}(xy^i)$ we may calculate: $$\sum_{k=0}^r P \mathcal{S}'_{r-k} (y^i; P \mathcal{S}_k(x)) \text{ for all } r.$$

By considering $\mathbb{H}_*(\mathcal{L}S^n)$, we see that $x = c^*(x_n)$ where $x_n \in H^n(S^n)$ is the generator. Using Corollary \ref{corollary:qsseqpopintertwine2}, $$P \mathcal{S} \circ c^*(x_n) = c^*_{eq} \circ Sq(x_n)$$ and classically $Sq(x_n) = x_n h^n$. Recall that $c^*_{eq} = c^* + h(\ldots)$, so $P \mathcal{S}(x) = x h^n + h^{n+1}(\ldots)$. Specifically there is a representative of $P \mathcal{S}(x)$ such that $P \mathcal{S}_n(x)$ represents $x$ and $P \mathcal{S}_r(x) = 0$ for $r < n$. This implies that for $r=n$: $$P \mathcal{S}_n(x y^i) = \sum_{k=0}^n P \mathcal{S}'_{n-k}(y^i;P \mathcal{S}_k (x)) = P \mathcal{S}'_0(y^i;x) = x y^{2i},$$ and for $r < n$, $$P \mathcal{S}_r(x y^i) = \sum_{k=0}^{r} P \mathcal{S}'_{r-k}(y^i; P \mathcal{S}_k(x)) = 0.$$ This is as much as we may prove for these specific representatives of the generators (in fact, we have not even shown that $P \mathcal{S}(x y^i) \neq 0$, because we have not ruled out $x y^i h^n + h^{n+1}(\ldots)$ being exact).

The other generators of $SH^*(T^* S^n)$ correspond to $y^i$ (for $i \in \mathbb{N})$. Recall that $P\mathcal{S}_0(y^i)$ is a chain representative of $y^i * y^i =  y^{2i}$. Hence $P \mathcal{S}(y^i) = y^{2i} + h(\ldots)$, abusively denoting by $y^{2i}$ our chain level representative of $y^{2i}$. From this we can prove slightly more, specifically that:

\begin{lemma}
$P \mathcal{S}(y^i) \neq 0$.
\end{lemma}
\begin{proof}
Suppose for a contradiction that $P \mathcal{S}(y^i) = y^{2i} + h(\ldots) $ is a boundary, so $y^{2i} + h(\ldots) = d_{eq}(A_0 + A_1 h^1 + \ldots)$. This implies that $y^{2i} = d(A_0)$ ($d$ here being the nonequivariant differential), hence $y^{2i}$ is exact. This is a contradiction from the ring structure of $SH^*(T^* S^n)$.
\end{proof}

\begin{rmk}

It should be reiterated that $P \mathcal{S}_i$ and $P \mathcal{S}'_{j}$ are defined on the chain level. They are not in general maps on homology, even though $P \mathcal{S}$ and $P \mathcal{S}'$ are well defined on homology.

\end{rmk}

\section{The Equivariant Pair-of-Pants for Negative Line Bundles}
\label{sec:eqpopnlb}
In this section we extend the work by Ritter in \cite{ritterFTNLB} to the $\mathbb{Z}/2$-equivariant setting. Specifically, we are interested in the total space of a negative line bundle $E$ over a closed symplectic manifold $B$, which we denote $M = \text{Tot}(E \rightarrow B)$. We are interested in the case where $M$ is monotone, and we use a Novikov field $\Lambda = \mathbb{Z}/2 ((T))$, where $|T| = 2N$. The minimal Chern number $N$ in such a case is defined by $c_1(\pi_2(M)) = N \mathbb{Z}$ for $N> 0$.  We will only repeat immediately relevant technical details of the cited paper, in order to minimise repitition.

\subsection{The relation between the quantum Steenrod square and the symplectic square}
We recall that Ritter proved in \cite[Theorem 1, see Section 1.3]{ritterFTNLB}, with $M$ as above, that there is a particular linear homomorphism $r: QH^*(M) \rightarrow QH^{*+2}(M)$ such that \begin{equation} \label{equation:ritterFTNLB} QH^*(M)/(\ker r^k)\xrightarrow{\cong} SH^*(M), \end{equation} for $k \gg 0$, with the isomorphism being induced by the $c^*$ map. Indeed, $\ker r^k = \ker c^*$ and we denote $K = \ker r^k$. 

For the case of $\mathbb{Z}/2$-equivariant Floer theory, starting with the commuting square in Equation \eqref{equivcstar}, we may write down the following commutative diagram:

\begin{equation}\label{equivFTNLB}
\xymatrix{
QH^*(M)/ K \cdot \Lambda
\ar@{->}^-{Q\mathcal{S}}[r]
\ar@{->}^-{\cong}_{c^*}[d]
&
QH^*_{eq}(M)/ Q \mathcal{S}(K) \cdot \Lambda[[h]]
\ar@{->}_-{c^*_{eq}}[d]
\\ 
SH^*(M)
\ar@{->}^-{P \mathcal{S}}[r]
&
SH^*_{eq}(M)
}
\end{equation}

The fact that $Q \mathcal{S}$ descends to the top horizontal map of Equation \eqref{equivFTNLB} is by linear algebra. The fact that $c^*_{eq}$ descends to the right hand vertical map of Equation \eqref{equivFTNLB} comes from the commutativity of \eqref{equivcstar}, i.e. \begin{equation} \label{equation:rqscommute} c^*_{eq}(Q \mathcal{S}(k)) = P \mathcal{S}(c^*(k)) = 0, \end{equation} for any $k \in K$. In general we do not know any more about the homomorphism $c^*_{eq}$, but we will discuss the setup of the problem. In the next section we will prove that $c_{eq}^*$ is an isomorphism for $\text{Tot}(\mathcal{O}(-1) \rightarrow \mathbb{CP}^n)$. 

We proceed as in \cite{ritterFTNLB}, we choose a loop of Hamiltonian symplectomorphisms $g_t: M \rightarrow M$ for $t \in \mathbb{R}/\mathbb{Z}$ such that $g_t$ is multiplication by $e^{2 \pi i t}$ (rotation of the fibre $\mathbb{C}$ of $E$). The $g_t$ are generated by the Hamiltonian $H_1 = H_1(R) = (1+ \epsilon) R$ for some $0 < \epsilon \ll 1$ depending on the choice of symplectic form. Here $R$ is the radial coordinate of the line bundle. Let $H_k(R) = k(1+\epsilon) R$. Define for $t \in \mathbb{R}/2\mathbb{Z}$ \begin{equation} \label{equation:HwrappingG} g^* H_k = H_k \circ g_t - H_1 \circ g_t = H_{k-1},\end{equation} as $g_t$ preserves $R$, and \begin{equation} \label{equation:JwrappingG} (g^*J_{eq,v})_t = d(g_t)^{-1} \circ J_{eq,v,t} \circ d(g_t), \end{equation} which still satisfies the conditions of Section \ref{subsec:equivHF}, remaining regular as in \cite[Theorem 18]{ritterFTNLB}. Notice that $g_t$ is $1$-periodic, so for $t \in \mathbb{R}/2\mathbb{Z}$ it is twice wrapped. In particular, \eqref{equation:JwrappingG} ensures that $(g^* J_{eq,-v})_{t+1} = (g^* J_{eq,v})_t$. There is a choice of lift of $g_t$ to $\tilde{g}$, an action on the cover $\widetilde{\mathcal{L}_0 M} \rightarrow \mathcal{L}_0 M$, where $\mathcal{L}_0 M$ consists of contractible free loops on $M$ and $\widetilde{\mathcal{L}_0 M}$ consists of pairs $(v,x)$ where $x \in \mathcal{L}_0 M$ and $v: D^2 \rightarrow M$ with $\partial v = x$, up to a relation $\sim$ where $(v_1,x_1) \sim (v_2,x_2)$ exactly when $x_1 = x_2$ and $\omega, c_1$ both vanish on $v_1 \# \overline{v_2}$. In the case of this specific $g_t$ we choose $\tilde{g}_t$ as in \cite[Section 7.8]{ritterFTNLB}, such that $\tilde{g} \cdot (c_x,x) = (c_x,x)$ where $x \in \mathcal{L}_0 M$ is a constant loop in the zero section of $M$ and $c_x: D^2 \rightarrow M$ is the constant map at $x$.

We define $$S_{\tilde{g}}^{eq} : CF^*_{eq}(2 \cdot H, J_{eq}) \rightarrow CF_{eq}^{*+4}(2 \cdot g^*H, g^* J_{eq}), \quad c \mapsto \tilde{g}^{-1}c.$$ Specifically, let $c = (x,\tilde{x}) \in CF^*(2 \cdot H)$, where we consider $x$ as being $2$-periodic with respect to $H$ (Section \ref{subsec:equivHF}) and $\tilde{x}$ is a filling disc for $x$. Then $\tilde{g}^{-1}(x,\tilde{x}) = (y, \tilde{y})$ where $\tilde{y}(t) = \tilde{g}^{-1}_t(\tilde{x}(r,t))$, using polar coordinates on the disc for $t \in \mathbb{R}/2\mathbb{Z}$. The grading shift $* \mapsto *+ 4$ comes from the Maslov index of $(g^2, \tilde{g}^2)$: see \cite[Section 3.1 and Lemma 48]{ritterFTNLB}. Further, this commutes with the equivariant continuation maps using the equivariant gluing arguments as have appeared throughout this paper, so \begin{equation} \label{equation:commuteact} S^{eq}_{\tilde{g}} \circ \Phi_{eq,H,H'} = \Phi_{eq, g^* H, g^*H'} \circ S^{eq}_{\tilde{g}}.\end{equation} 

Observe that, for $r_{eq}$ defined by the composition $$r_{eq} = \Psi_{eq}^{-1} \circ \Phi_{eq,H_{-1},H_0} \circ S^{eq}_{\tilde{g}} \circ \Psi_{eq},$$ the following diagram commutes:

\begin{equation}\label{dirlimetc}
SH_{eq}^*(M) \xymatrix{
=\underset{H}{\varinjlim}\left( HF^*_{eq}(H_0) \right.
\ar@{->}^-{\Phi_{eq,H_0,H_1}}[rr]
&
&
HF^*_{eq}(H_1)
\ar@{->}^-{\Phi_{eq,H_1,H_2}}[rr]
& 
&
HF^*_{eq}(H_2)
\ar@{->}^-{\Phi_{eq,H_2,H_3}}[r]
&
\left. \ldots \right)
\\ 
=\underset{H}{\varinjlim}\left( QH^*_{eq}(M) \right.
\ar@{->}_-{r_{eq}}[rr]
\ar@{->}^{\cong}_-{\Psi_{eq}}[u]
&
&
QH^{*+4}_{eq}(M)
\ar@{->}_-{r_{eq}}[rr]
\ar@{->}^{\cong}_-{(S^{eq}_{\tilde{g}})^{-1} \circ \Psi_{eq}}[u]
&
&
QH^{*+8}_{eq}(M)
\ar@{->}_-{r_{eq}}[r]
\ar@{->}^{\cong}_-{(S^{eq}_{\tilde{g}})^{-2} \circ \Psi_{eq}}[u]
&
\left. \ldots \right)
} 
\end{equation}

Hence we may determine $SH^*_{eq}(M)$ using the lower row. Compare to \cite[Section 3, Section 4.2]{ritterFTNLB}, where the map $$S_{\tilde{g}}: HF^*(H,J) \rightarrow HF^{*+2}(g^* H, g^* J), \quad c \mapsto g^{-1} c,$$ was similarly defined, where $c$ is a $1$-periodic Hamiltonian loop for $H$ and $(g c)(t) = g_t \cdot c(t)$. As $1$-periodic loops were used in \cite{ritterFTNLB}, the index change is $2$, the Maslov index of $(g,\tilde{g})$. Likewise the map $r: QH^*(M) \rightarrow QH^{*+2}(M)$ was defined by the composition \begin{equation} \label{equation:nonequr} r = \Psi^{-1} \circ \Phi_{H_{-1},H_0} \circ S_{\tilde{g}} \circ \Psi.\end{equation} Ritter then showed that there is some $k>0$ such that $r|_{\text{Im} \ r^k}: \text{Im} \ r^k \rightarrow \text{Im} \ r^{k+1}$ is an isomorphism, hence $c^* : QH^*(M) / \ker r^k \xrightarrow{\cong} SH^*(M)$.  Unfortunately, the methods in that paper do not apply imediately, because $QH^*_{eq}(M)$ is not a finite dimensional $\Lambda$-module.

\begin{lemma}
\label{lemma:finalthm0}
$r_{eq} \circ Q \mathcal{S} = Q \mathcal{S} \circ r$.
\end{lemma}
\begin{proof}
\begin{equation}\begin{array}{rcl}
Q \mathcal{S} \circ r & = & Q \mathcal{S} \circ \Psi^{-1} \circ \Phi_{H_{-1},H_0} \circ S_{\tilde{g}} \circ \Psi \\

 &=& \Psi_{eq}^{-1} \circ P \mathcal{S} \circ \Phi_{H_{-1},H_0} \circ S_{\tilde{g}} \circ \Psi \\

 &=& \Psi_{eq}^{-1} \circ \Phi_{eq, H_{-1},H_0} \circ P \mathcal{S} \circ S_{\tilde{g}} \circ \Psi \\
\end{array}\end{equation}
using respectively Theorem \ref{thm:qsseqpopintertwine} and Lemma \ref{lemma:eqcontinmaps}. To see that $$P \mathcal{S} \circ S_{\tilde{g}} = S^{eq}_{\tilde{g}} \circ P \mathcal{S},$$ observe that if we count the coefficient of $y$ in $S^{eq}_{\tilde{g}} \circ P \mathcal{S}(x)$ then we count $u: S \rightarrow M$ satisfying \eqref{equation:eqPOP} except that we replace the conditions $y(t)$ and $y(t+1)$ by $g_t \cdot y(t)$ and $g_t \cdot y(t+1)$ respectively. Similarly for $x(t)$ being replaced by $g_t x(t)$ in \eqref{equation:eqPOP} for $P \mathcal{S} \circ S_{\tilde{g}}(x)$.

Let $(g^* u)(z) = g^{-1}_{\pi_2 z} u(z)$ , where the $2$ to $1$ branched cover $S \rightarrow \mathbb{R} \times \mathbb{R}/\mathbb{Z}$ uses the projection map $\pi = \pi_1 \times \pi_2$. Then there is a bijective correspondence between pairs $(w,u)$ that one counts for the coefficient of $y$ in $S^{eq}_{\tilde{g}} \circ P \mathcal{S}(x)$ and solutions $(w, g^* u)$ counted for the coefficient of $y$ in $P \mathcal{S} \circ S_{\tilde{g}}(x)$. Here, if we use the data $(H, J^w_z, Y)$ for $S^{eq}_{\tilde{g}} \circ P \mathcal{S}$ then we use the data $(g^* H, g^* J^w_z, g^* Y)$ for $P \mathcal{S} \circ S^{eq}_{\tilde{g}}$, where $$g^* J^w_z := d(g_{\pi_2 z})^{-1} \circ J^w_z \circ dg_{\pi_2 z}$$ satisfies the relevant conditions in Section \ref{subsec:eqpop}.
\end{proof}

\begin{rmk}
The Maslov index above uses $g^2$, because $g_t$ wraps twice for $t \in \mathbb{R}/2\mathbb{Z}$. We recall that in the space of based loops in a Lie group, the product induced by the group action is homotopic to the product induced by composition of loops.
\end{rmk}

\begin{rmk}
By Lemma \ref{lemma:finalthm0} above, we see that $$Q \mathcal{S} (\ker r^k) \subset \ker r^k_{eq}.$$ This gives another proof that \eqref{equation:rqscommute} holds.
\end{rmk}

\subsection{The symplectic square for $M = \text{Tot}(\mathcal{O}(-1) \rightarrow \mathbb{CP}^m)$}
Observe first that $M = \text{Tot}(\mathcal{O}(-1) \rightarrow \mathbb{CP}^m)$ deformation retracts onto $\mathbb{CP}^m$, hence they have the same cohomology. Therefore, all that is strictly different between the quantum cohomologies of $M$ and $\mathbb{CP}^m$ is the interaction with $J$-holomorphic spheres. Recall that $|T| = 2m$, where $T$ is the quantum variable.

We state but do not reprove from \cite[Theorem 61]{ritterFTNLB} that \begin{equation} \label{equation:qhstarm} QH^*(M) = \Lambda[\omega_Q] / (\omega_Q^{m+1} + T \cdot \omega_Q), \end{equation} where $\omega_Q$ is the symplectic form on $\mathbb{CP}^m$. We will distinguish by $\omega^i$ and $\omega_Q^i$ taking $\omega \cup ... \cup \omega$ and $\omega * ... * \omega$, respectively the $i$-th power for the cup and the quantum cup products. Hence in this notation $\omega = \omega^1 = \omega_Q^1$. 

We will give a way to iteratively compute $Q \mathcal{S}$, in the same way as was used in \cite[Section 6]{me}. We refer the reader to the aforementioned section for more details. We use the quantum Cartan relation from \cite[Theorem 1.2]{me} to show that $$Q \mathcal{S}(\omega_Q^{i+1}) = Q \mathcal{S}(\omega_Q^{i}) * Q \mathcal{S}(\omega_Q) + \sum_{j,k} q_{j,k}(W_0 \times D^{j-2,+})(\omega_Q^{i},\omega_Q)h^{j}.$$ Using that $Q \mathcal{S}(\omega_Q) = \omega_Q * \omega_Q + \omega_Q h^2$ (which follows from a combination of basic properties of the Quantum Steenrod square), it is sufficient to calculate $q_{j,k}(W_0 \times D^{j-2,+})(\omega_Q^{i},\omega_Q)$ for each $i,j,k$ such that $0 \le i \le m$, and use the quantum Cartan relation.

In \cite[Lemma 5.6]{me}, it was proved that $q_{j,0} = 0$. For degree reasons $q_{j,k} = 0$ when both $k > 1$ and $i < m$. Hence, for $i<m$ we only need to consider $k=1$, i.e. a degree $1$ $J$-holomorphic sphere. This corresponds to setups of the form of Figure \ref{fig:totcpnquant} (see \cite[Section 5]{me} for justification). In the figure, the intersection conditions with ``$PD(\omega_Q^i)_v$" is shorthand for the conditions of evaluating the holomorphic sphere at some chosen pseudocycle representatives of $PD(\omega_Q^i)_v$ as in Definition \ref{defn:eqpopqss}. One can show that because $\text{codim} \ PD(\omega_Q) = 2$ which is the dimension of the $J$-holomorphic sphere, the intersection conditions with $\omega_Q$ are unnecessary except for fixing marked points. Thus, the calculation of the coefficient of $\omega_Q^l \cdot T$ in $q_{j,1}(W_0 \times D^{j-2,+})(\omega_Q^{i},\omega_Q)$ is counting how many degree $1$ $J$-holomorphic spheres there are intersecting some generic pseudocycle representatives of $PD(Sq^j(\omega_Q^i)) \in H_*(D, \partial D)$ and $PD(\omega_Q^l)^{\vee} \in H_*(D)$, where $\vee$ is the intersection duality. For degree reasons there can only be solutions for $(j,l) = (2(m-i),2), \ (2(m-1-i),0)$. By moving the $\mathbb{CP}^0$ (i.e. the point representing $PD(\omega_Q^l)^{\vee}$) in the $l=0$ case to infinity (see \cite[Theorem 61]{ritterFTNLB}) we see the only possibly non-zero term occurs for $(j,l) = (2(m-i),2)$. Hence \begin{equation} \label{equation:qssfortot} \sum_{j,k} q_{j,k}(W_0 \times D^{j-2,+})(\omega_Q^{i},\omega_Q)h^{j} = {{i}\choose{m-i}} \omega_Q T h^{2+4i-2m}, \end{equation} where ${{i}\choose{m-i}}$ is the coefficient of $\omega_Q^m$ in $Sq^{m-i}(\omega_Q^i)$ (the classical Steenrod square is a homotopy invariant, and we know the answer for $\mathbb{CP}^m$). For $i=m$, all of the previous holds for $k = 1$, but there is also the possibility of $k=2$. If $k=2$ then $j=2$, and recall that $j=2$ corresponds to using $D^{0,+}$ (a point) as the parameter space. Hence this is the $T^2$ term in the standard quantum product $\omega_{Q}^{m}*\omega_{Q}^m*\omega_{Q} h^2$ (it is not $\omega_{Q}^{m}*\omega_{Q}^m*\omega_{Q}*\omega_{Q}$, because the final intersection of the holomorphic sphere with the hypersurface $PD(\omega_{Q}^1)$ fixes the marked point associated to varying our domains in $W_0$), giving that $$\sum_{j,k} q_{j,k}(W_0 \times D^{j-2,+})(\omega_Q^{m},\omega_Q)h^{j} = \omega_Q T h^{2+2m} + \omega_{Q}T^2 h^2.$$ These can be used to calculate all of the quantum Steenrod squares for any $m$. 

\begin{exmpl}
\label{exmpl:PSexmpl}
We let $x = \omega_Q$, and $M = \text{Tot}(\mathcal{O}(-1) \rightarrow \mathbb{CP}^4)$.

$$\begin{array}{ll} Q \mathcal{S}(x) & = x^2 + x h^2, \\
 Q \mathcal{S}(x^2) &= Q \mathcal{S}(x) * Q \mathcal{S}(x) +  {{1}\choose{3}} x T h^{-2} \\ 
& = (x^2 + xh^2)^2 = x^4 + x^2 h^4, \\
 Q \mathcal{S}(x^3) &= Q \mathcal{S}(x^2) * Q \mathcal{S}(x) +  {{2}\choose{2}} x T h^{2} \\ 
& = (x^2 + xh^2)*(x^4 + x^2 h^4) + xT h^2 \\ & =  x^2 T + x^4 h^4 + x^3 h^6, \\
 Q \mathcal{S}(x^4) &= Q \mathcal{S}(x^3) * Q \mathcal{S}(x)+  xTh^6 \\ 
&=  x^4 T +  x^3 T h^2 +  x^2 T h^4 +  x T h^6 + x^4 h^8, \\
Q \mathcal{S}(x^5) &= Q \mathcal{S}(x^4) * Q \mathcal{S}(x) + x T h^{10} + x T^2 h^2 \\ 
&= (x ^2 + x h^2) T^2 = Q \mathcal{S}(xT),\end{array}$$

remembering that in this case there is the extra correction term $x T^2 h^2$.
\end{exmpl}

		\begin{figure}
			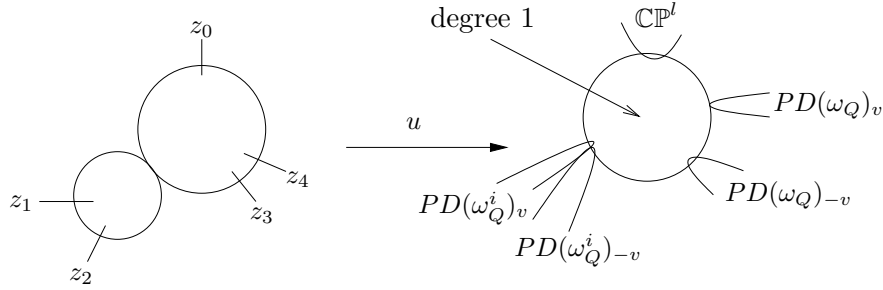
			\caption{Setups for the correction term in calculating $Q \mathcal{S}$ for $M$.}
			\label{fig:totcpnquant}
		\end{figure}

For this choice of $M$, the linear map $r$ is multiplication by the symplectic form $x$ (see e.g. \cite[Lemma 60]{ritterFTNLB}). In general it is multiplication by the Seidel element. Then $\ker r = x^n + T$, and $r |_{\text{Im} \ r}: \text{Im} \ r \rightarrow \text{Im} \ r^2$ is an isomorphism. We denote \begin{equation} \label{equation:ridegrees} r_{eq} = \sum_{i \ge 0} h^{2i} r_i,\end{equation} where $r_i: QH^*(M) \rightarrow QH^{*+4-2i}(M)$. Observe that $r_0$ uses constant flowlines on $S^{\infty}$, hence $r_0 = r^2$ where $r$ is as in Equation \eqref{equation:nonequr}.

We note that:
\begin{itemize}
\item we have calculated the correction terms of the quantum Cartan relation above, 
\item we can show that $Q \mathcal{S}(x^m) = x^m h^{2m} + \ldots (\text{lower order in h})$, because $Q \mathcal{S}(x) = x h^2 + x^2$, then we use the quantum Cartan relation and induction to show \begin{equation} \label{equation:firsttermqsxi} Q \mathcal{S}(x^i) = x^i h^{2i} + \ldots (\text{lower order in h}), \end{equation} for each $i = 0,...,m$.
\item from Equation \eqref{equation:qhstarm}, we know that $x^{m+1} = x T$ 
\item hence $Q \mathcal{S}(x^{m+1}) = Q \mathcal{S}(xT) = (x^2 + xh^2)T^2$, 
\end{itemize}
we can deduce, using the quantum Cartan relation, the points above and dividing $Q \mathcal{S}(x^{m+1}) + x Th^{2m+2} + x T^2 h^2$ by $Q \mathcal{S}(x) = x^2 + x h^2$, that \begin{equation} \label{equation:qsxm} Q \mathcal{S}(x^m) = x^m h^{2m} + T \left( \sum_{i=1}^m x^i h^{2m - 2i} \right). \end{equation}

\begin{proof}[Proof of Theorem \ref{thm:finalthm}]
We must prove two lemmas.

\begin{lemma}
\label{lemma:finallemma1}
$$r_{eq} |_{\text{Im} \ r_{eq}}: \text{Im} \ r_{eq} \rightarrow \text{Im} \ r_{eq}^{2}$$ is an isomorphism. 
\end{lemma}
\begin{proof}
The $\Lambda$ vector space $QH^*(M)$ is generated by $\{ x, \ldots , x^m, x^m + T \}$. The additive group $\text{Im }r = r(QH^*(M)) = x* QH^*(M)$ is generated by $\{ x, \ldots , x^m \}$, and $x^m + T$ generates $K = \ker r$. We will prove that $r_i(QH^*(M)) \subset \text{Im } r$ for all $i$. Observe that for $x \in H^*(M)$, we have $r_i(x) = \sum_{j \ge 0} r_{i,j}(x) T^j$, where $r_{i,j}(x) \in H^{*+4-2i - 2jN}(M)$. As the cohomology is bounded below by degree $0$ we see that $r_i = 0$ for $2i > \dim M + 4$, hence only a finite number of $r_i$ are nonzero.

There are five classes of cases to check (all other $r_i$ must land in $\text{Im} \ r$ for degree reasons: specifically, $r_i (x^n)$ may land in $\text{Im} \ r$ only if $|r_i(x^n)|$ is divisible by $m$. These cases are:
\begin{enumerate}
\item $r_0 (x^{m-2})$,
\item $r_2(1)$,
\item $r_{i+2}(x^i)$ for $i=1,...,m$,
\item $r_2(x^m)$,
\item $r_1(x^{m-1})$.
\end{enumerate}

Case $(1)$ is immediate, as $r_0 = r^2$. For Case $(2)$, Lemma \ref{lemma:finalthm0} implies that $$r_{eq}(1) = r_{eq}(Q \mathcal{S}(1)) = Q \mathcal{S}(r(1)) = Q \mathcal{S}(x) = x^2 + xh^2.$$ Hence $r_2(1) = 0$. For Case $(3)$, we use downwards induction. For the base case, Lemma \ref{lemma:finalthm0} implies that \begin{equation} \label{equation:vanishinQSK} r_{eq}(Q \mathcal{S}(x^m + T)) = Q \mathcal{S}(r(x^m+T)) = 0. \end{equation} Using Equation \eqref{equation:qsxm}, the $h^{4m+4}$ term of the left hand side of Equation \eqref{equation:vanishinQSK} is $r_{m+2}(x^m)$. The $h^{4m+4}$ term of the right hand side is $0$, hence $$r_{m+2}(x^m) = 0.$$ For the induction step we proceed similarly. Lemma \ref{lemma:finalthm0} implies that \begin{equation} \label{equation:inductivestuff} r_{eq}(Q \mathcal{S}(x^i)) = Q \mathcal{S}(x^{i+1}). \end{equation} We rewrite Equation \eqref{equation:firsttermqsxi} as: \begin{equation} \label{equation:somecoeffsfinal} Q \mathcal{S}(x^i) = x^i h^{2i} + \sum_{j=1}^{m-i} \epsilon_j x^{i+j} h^{2i-2j} + \delta T h^{4i-2m} + \sum_{j \ge 1} \eta_j T x^j h^{4i-2m-2j},\end{equation} for some $\epsilon_j, \delta, \eta_j \in \mathbb{Z}/2$. Hence the $h^{4i+4}$ term on the left hand side is \begin{equation} \label{equation:qsxi} r_{i+2}(x^i) + \sum_{j=1}^{m-i} \epsilon_j r_{i+ j+2}(x^{i+j}) + \delta T r_{m+2}(1) + \sum_{j \ge 1} T r_{m+2+j}(x^j).\end{equation} Using Equation \eqref{equation:ridegrees}, as $m \ge 1$, for degree reasons $r_{m+2+j}(x^j) = 0$ for all $j \ge 0$. By the induction hypothesis, we assume that $r_{i+ j+2}(x^{i+j})  = 0$ for all $1 \le j \le m-i$. Hence the expression in \eqref{equation:qsxi} reduces to $r_{i+2}(x^i)$. We also note that the $h^{4i+4}$ term on the right hand side of Equation \eqref{equation:inductivestuff} is $0$ using Equation \eqref{equation:somecoeffsfinal}. Hence by induction $r_{i+2}(x^i) = 0$.

For Case $(4)$, using Equation \eqref{equation:qsxm} to expand the $h^{2m+4}$ term of \eqref{equation:vanishinQSK} implies that \begin{equation} \label{equation:vanishinr2xm} r_2(x^m)  + T \sum_{i = 1}^m r_{i+2}(x^i) + T^2 r_{m+2}(1)= 0. \end{equation} We know from the previous cases that all of the terms except $r_2(x^m)$ on the left hand side of \eqref{equation:vanishinr2xm} vanish, hence $r_2(x^m) = 0$.

For the final case, consideration of the $h^4$ term of \eqref{equation:vanishinQSK} implies that $$T( r_2(x^m + T) + r_1(x^{m-1}) + r_0(x^{m-2}) )= 0.$$ We know that $r_2(x^m+T) = 0$ from cases $(2)$ and $(4)$, hence $$r_1(x^{m-1}) = r_0(x^{m-2}) = r^2(x^{m-2}) = x^m.$$

Knowing that $r |_{\text{Im} \ r} : \text{Im} \ r \xrightarrow{\cong} \text{Im} \ r^2 \subset \text{Im} \ r$ is an isomorphism, there exists $a:  \text{Im} \ r^2 = \text{Im} \ r \rightarrow \text{Im} \ r$ such that $a r^2 = r$. Use an argument as in Lemma \ref{lemma:classicalpowerserieslemma} to iteratively construct some $a_{eq}: r(QH^*(M))[[h]] \rightarrow r(QH^*(M))[[h]]$, of the form $$a_{eq} = a^2 + h^2 a^2 r_1 a^2 + \sum_{i \ge 2} h^{2i} a_i,$$ such that the $$a_i: \text{Im} \ r^2 = \text{Im} \ r \rightarrow \text{Im} \ r$$ are written as compositions of $a$ and the $r_i$, and $a_{eq} r^2_{eq} = r_{eq}$. We required that $r_i(QH^*(M)) \subset r(QH^*(M)) = r^2(QH^*(M))$ for all $i$ so that $a_{eq}$ is well defined (specifically, it ensures that $a^k r^k r_i = r_i$).
\end{proof}

\begin{lemma}
\label{lemma:finallemma2}
$$\ker r_{eq} = Q \mathcal{S}(x^n + T) \cdot \Lambda[[h]].$$
\end{lemma}
\begin{proof}
Suppose $x= \sum_{i \ge 0} x_i h^i \neq 0$ is some element of $\ker r_{eq}$. Let $j$ be minimal such that $x_j \neq 0$. Observe that $r_{eq}(x) = 0$, which implies that $r^{2}(x_j) = 0$, i.e. $x_j \in \ker r^2 = \langle x^n + T \rangle$. Hence $$x_j = \lambda_j \cdot ( x^n + T),$$ for $\lambda_j \in \Lambda$. Replace $x$ by $x- h^j \lambda_j T^{-1}Q \mathcal{S}(x^n + T)$ and note that the coefficient of $h^i$ in $x- h^j \lambda_j T^{-1}Q \mathcal{S}(x^n + T)$ is zero for $i \le j$. We then iterate. This yields that $$x =\left( T^{-1} \sum_{i\ge 0} \lambda_{i} h^i \right) Q \mathcal{S}(x^n + T),$$ for some $\lambda_i \in \Lambda$.
\end{proof}

Using an argument as in \cite[Section 4.2]{ritterFTNLB} and the diagram \ref{dirlimetc}, by Lemma \ref{lemma:finallemma1} we deduce that $$c^*_{eq}: QH^*_{eq}(M) / \ker r_{eq} \xrightarrow{\cong} SH^*_{eq}(M).$$ Then Theorem \ref{thm:finalthm} follows by Lemma \ref{lemma:finallemma2}. 
\end{proof}

We finish by calculating $P \mathcal{S}$ for the same case as Example \ref{exmpl:PSexmpl}, using Equation \eqref{equivFTNLB}, and observing by the work earlier in this subsection that we can do a similar calculation for all $\text{Tot}(\mathcal{O}(-1) \rightarrow \mathbb{CP}^n)$. Recall that in general $SH^*_{eq}(M)$ does not have an obvious ring structure, but in this case it inherits one from $QH^*_{eq}(M)$.

\begin{exmpl}
We let $x = \omega_Q$, and $M = \text{Tot}(\mathcal{O}(-1) \rightarrow \mathbb{CP}^4)$, using Theorem \ref{thm:finalthm}, we know that $$SH_{eq}^*(M) \cong \Lambda[x][[h]]/ Q \mathcal{S}(x^4+ T) \cdot \Lambda[[h]].$$ Then using Example \ref{exmpl:PSexmpl} we calculate:

$\begin{array}{ll} P \mathcal{S}(x) & = x^2 + x h^2  \end{array}, $ \\ 

$\begin{array}{ll} P \mathcal{S}(x^2) & = x^4 + x^2 h^4 \\ &= T + x^3 h^2 + x h^6 + {\displaystyle \sum_{i \ge 0}} h^{8i}T^{-i}(T + x^3 h^2 + x^2 h^4 +  xh^6), \end{array}$ \\ 

$\begin{array}{ll} P \mathcal{S}(x^3) & = x^2 T + x^4 h^4 + x^3 h^6 \\ &= x^2 T + Th^4 + x^2h^8 + xh^{10}+ {\displaystyle \sum_{i \ge 0}} h^{4+8i}T^{-i}(T + x^3 h^2 + x^2 h^4 +  xh^6),  \end{array}$ \\ 

$\begin{array}{ll}  P \mathcal{S}(x^4) &= x^4 T +  x^3 T h^2 +  x^2 T h^4 +  x T h^6 + x^4 h^8 \\ &= T^2 = P \mathcal{S}(T) , \end{array}$

\end{exmpl}

\begin{rmk}
Lemma \ref{lemma:finallemma2} works more generally for $$M = \text{Tot}(\mathcal{O}(-k) \rightarrow \mathbb{CP}^m),$$ when $k \le (m+1)/2$ is odd. That is because in this case, $$QH^*(M) \cong \Lambda [x] / (x^{m+1} + T x^k),$$ $r$ is quantum product with $x$ and $$(x^{m-k+1} + T) = \ker r^k.$$ For the given range of $k$, we have that $(x^{m-k+1} + T)^2 = T(x^{m-k+1} + T)$, hence it generates $\ker r^k$. That is the key idea in the proof of the lemma.

It should also be reasonable that Lemma \ref{lemma:finallemma1} works for the given range of $k$, although this has not been proven. More generally, there is no guarantee that either lemma holds more generally for the total space of a negative line bundle over a closed symplectic manifold, or even for odd $k> (m+1)/2$. For even $k$ the symplectic cohomology is trivial.
\end{rmk}
\bibliographystyle{plain}
\bibliography{biblio}

\begin{thebibliography}{10}

\bibitem{buonhacon}
Sandro Buoncristiano and Derek Hacon.
\newblock An elementary geometric proof of two theorems of {T}hom.
\newblock {\em Topology}, 20(1):97--99, 1981.

\bibitem{floerfixedpoints}
Andreas Floer.
\newblock Symplectic fixed points and holomorphic spheres.
\newblock {\em Comm. Math. Phys.}, 120(4):575--611, 1989.

\bibitem{morcombound}
Fran\c{c}ois Laudenbach.
\newblock A {M}orse complex on manifolds with boundary.
\newblock {\em Geom. Dedicata}, 153:47--57, 2011.

\bibitem{jhols}
Dusa McDuff and Dietmar Salamon.
\newblock {\em {$J$}-holomorphic curves and quantum cohomology}, volume~6 of
  {\em University Lecture Series}.
\newblock American Mathematical Society, Providence, RI, 1994.

\bibitem{jholssympl}
Dusa McDuff and Dietmar Salamon.
\newblock {\em {$J$}-holomorphic curves and symplectic topology}, volume~52 of
  {\em American Mathematical Society Colloquium Publications}.
\newblock American Mathematical Society, Providence, RI, second edition, 2012.

\bibitem{oancea}
Alexandru Oancea.
\newblock Morse theory, closed geodesics, and the homology of free loop spaces.
\newblock In {\em Free loop spaces in geometry and topology}, volume~24 of {\em
  IRMA Lect. Math. Theor. Phys.}, pages 67--109. Eur. Math. Soc., Z\"urich,
  2015.
\newblock With an appendix by Umberto Hryniewicz.

\bibitem{PSS}
S.~Piunikhin, D.~Salamon, and M.~Schwarz.
\newblock Symplectic {F}loer-{D}onaldson theory and quantum cohomology.
\newblock In {\em Contact and symplectic geometry ({C}ambridge, 1994)},
  volume~8 of {\em Publ. Newton Inst.}, pages 171--200. Cambridge Univ. Press,
  Cambridge, 1996.

\bibitem{ritternov}
Alexander~F. Ritter.
\newblock Novikov-symplectic cohomology and exact {L}agrangian embeddings.
\newblock {\em Geom. Topol.}, 13(2):943--978, 2009.

\bibitem{ritterale}
Alexander~F. Ritter.
\newblock Deformations of symplectic cohomology and exact {L}agrangians in
  {ALE} spaces.
\newblock {\em Geom. Funct. Anal.}, 20(3):779--816, 2010.

\bibitem{ritterFTNLB}
Alexander~F. Ritter.
\newblock Floer theory for negative line bundles via {G}romov-{W}itten
  invariants.
\newblock {\em Adv. Math.}, 262:1035--1106, 2014.

\bibitem{salamonfloer}
Dietmar Salamon.
\newblock Lectures on {F}loer homology.
\newblock In {\em Symplectic geometry and topology ({P}ark {C}ity, {UT},
  1997)}, volume~7 of {\em IAS/Park City Math. Ser.}, pages 143--229. Amer.
  Math. Soc., Providence, RI, 1999.

\bibitem{equivcompactsal}
Dietmar~A. Salamon.
\newblock Quantum products for mapping tori and the {A}tiyah-{F}loer
  conjecture.
\newblock In {\em Northern {C}alifornia {S}ymplectic {G}eometry {S}eminar},
  volume 196 of {\em Amer. Math. Soc. Transl. Ser. 2}, pages 199--235. Amer.
  Math. Soc., Providence, RI, 1999.

\bibitem{schwarzmorsesingiso}
Matthias Schwarz.
\newblock Equivalences for {M}orse homology.
\newblock In {\em Geometry and topology in dynamics ({W}inston-{S}alem, {NC},
  1998/{S}an {A}ntonio, {TX}, 1999)}, volume 246 of {\em Contemp. Math.}, pages
  197--216. Amer. Math. Soc., Providence, RI, 1999.

\bibitem{seeliger}
Nora Seeliger.
\newblock Loop homology of spheres and complex projective spaces.
\newblock {\em Forum Math.}, 26(4):967--981, 2014.

\bibitem{seidel}
Paul Seidel.
\newblock The equivariant pair-of-pants product in fixed point {F}loer
  cohomology.
\newblock {\em Geom. Funct. Anal.}, 25(3):942--1007, 2015.

\bibitem{thom}
Ren\'e Thom.
\newblock Quelques propri\'et\'es globales des vari\'et\'es diff\'erentiables.
\newblock {\em Comment. Math. Helv.}, 28:17--86, 1954.

\bibitem{viterbo}
C.~Viterbo.
\newblock Functors and computations in {F}loer homology with applications. {I}.
\newblock {\em Geom. Funct. Anal.}, 9(5):985--1033, 1999.

\bibitem{me}
Nicholas Wilkins.
\newblock {A} {C}onstruction of the {Q}uantum {S}teenrod {S}quares and {T}heir
  {A}lgebraic {R}elations.
\newblock arXiv:1805.02438, 2018.

\end{thebibliography}

\end{document}